\documentclass[11pt,a4paper,reqno]{amsart}
\usepackage{mathrsfs}
\usepackage{amsmath,amssymb}
\usepackage{amsfonts}
\usepackage{latexsym,bm}
\usepackage{amsthm}
\usepackage{cite}
\usepackage{amssymb}
\usepackage{amssymb,amscd}
\usepackage{amsbsy}
\usepackage{fancyhdr,graphicx}
\usepackage{indentfirst}
\usepackage{graphics,color}
\usepackage{epsfig}
\usepackage{xcolor}
\usepackage{overpic}
\usepackage{caption}
\usepackage{subcaption}
\usepackage{multirow}
\usepackage{diagbox}
\usepackage{array}
\usepackage{makecell}
\usepackage{arydshln}
\usepackage{booktabs}
\usepackage{threeparttable}
\usepackage{extarrows}
\usepackage{enumerate}
\newtheorem {theorem} {Theorem}[section]
\newtheorem {proposition} [theorem]{Proposition}

\newtheorem {lemma}  [theorem]{Lemma}

\newtheorem {definition} [theorem]{Definition}
\newtheorem {remark} [theorem]{Remark}

\theoremstyle{remark}

\theoremstyle{definition}
\theoremstyle{conjecture}


\newpage

\begin{document}

\title[the limit cycles of population models with time-varying factors]
{On the study of the limit cycles for a class of population models with time-varying factors}

\author[R. Tian, J. Huang and Y. Zhao]
	{Renhao Tian$^1$, Jianfeng Huang$^{2}$, Yulin Zhao$^3$}


\address{$^1$ School  of Mathematics (Zhuhai), Sun Yat-sen University, Zhuhai Campus, Zhuhai, 519082,  China}
\email{tianrh@mail2.sysu.edu.cn}

\address{$^2$ Department of Mathematics, Jinan University, Guangzhou 510632, P. R. China}
\email{thuangjf@jnu.edu.cn}

\address{$^3$ School  of Mathematics (Zhuhai), Sun Yat-sen University, Zhuhai Campus, Zhuhai, 519082,  China}
\email{mcszyl@mail.sysu.edu.cn}



\begin{abstract}
    In this paper, we study a class of population models with time-varying factors, represented by one-dimensional piecewise smooth autonomous differential equations.
    We provide several derivative formulas in ``discrete'' form for the Poincar\'{e} map of such equations, and establish a criterion for the existence of limit cycles.
    These two tools, together with the known ones, are then combined in a preliminary procedure that can provide a simple and unified way to analyze the equations.
    As an application, we prove that a general model of single species with seasonal constant-yield harvesting can only possess at most two limit cycles, which improves the work of Xiao in 2016.
    We also apply our results to a general model described by the Abel equations with periodic step function coefficients, showing that its maximum number of limit cycles, is three.
    Finally, a population suppression model for mosquitos considered by Yu \& Li in 2020 and Zheng et al. in 2021 is studied using our approach.
\end{abstract}

\date{}

\maketitle
\section{Introduction and statements of main results}
The theory of ordinary differential equations is one of the powerful tools in the study of population biology. It provides effective ways to characterize the population dynamics and make practical predictions, and therefore has been extensively applied in modeling the evolutions of various species.
See e.g. \cite{biology2,biology1} for more information on this field.
For the problems of single-species populations, the application of this theory goes back to the essay of Malthus, and has become a rapidly developing field since the \emph{logistic} form was proposed (see e.g. \cite{logistic,biology2}). In this context, a single-species model is characterized by a one-dimensional autonomous differential equation
\begin{align}\label{eq0}
\frac{dx}{dt}=f(x),
\end{align}
where $x(t)$ and $f(x(t))$ represent the population size of the species and the rate of change at time $t$, respectively. The main particularity of such a model is that it is integrable and the behaviors of the orbits can be known by analyzing the zeros of $f$. This provides an opportunity to explore the population dynamics in a relatively simple way.
So far, a large number of single-species models of the form \eqref{eq0} have been established in the literature for many real biological problems, such as insect outbreaks, yeast growth, fish weight growth, harvesting, herding behavior, the Allee effect, and organ size evolution, see for instance \cite{sprucebudworm,book,harvest,Ben,logistic} and the references therein.

It should be noted that due to the time-invariant growth rates, models of the form \eqref{eq0} may not be able to accurately capture the population growth in all stages and therefore may deviate from the reality in the long term.
Indeed, it is not uncommon for the population growth rate of a species to be stable in the short term, but to change over different periods according to environmental variations.
For this reason, a natural way to overcome this imperfection is to modify the models by a kind of non-autonomous differential equations (actually, a kind of piecewise autonomous differential equations):
\begin{equation}\label{equation0}
\begin{split}
    \frac{dx}{dt}=f(t,x)=
    \left\{
  \begin{aligned}
        &f_{1}(x),&  t\in \big[kT+T_0,kT+T_{1}\big),\ x\in I,\\
        &f_{2}(x),&  t\in \big[kT+T_{1},kT+ T_{2}\big),\ x\in I,\\
        &\quad\cdots  \quad& \quad\cdots \quad\\
        &f_{n}(x),&  t\in \big[kT+ T_{n-1},kT+T_n\big),\ x\in I,
    \end{aligned}
  \right.
\end{split}
\end{equation}
where $k\in\mathbb{Z}$,
$n\in \mathbb{Z}_{\geq 2}$, $0=T_{0}<T_{1}<T_{2}< \cdots < T_{n-1}<T_{n}=T$, $I\subseteq\mathbb R$ is an interval, and $f_{i}$ is $C^{1}$-differentiable on $I$ with at most a finite number of zeros, representing the population growth rate during the sub-period $[T_{i-1}+kT,T_{i}+kT)$, $i=1,2,\cdots,n$.

In recent decades, there are quite a few applications of equation \eqref{equation0} in modeling population dynamics with time-varying environments, especially the seasonal variations and periodic anthropogenic factors.
See e.g. \cite{population2, population1, HsuZhao, population4, population3} and the references therein.
To specialize our next studies, we summarize some of the representative works in the following.

Hsu and Zhao \cite{HsuZhao} in 2012 considered a seasonal \emph{logistic} model throughout the year, in which the population of a species grows logistically during the breeding season, but declines due to adverse conditions in the bad season. That is,
\begin{equation}\label{eq1}
\begin{split}
    \frac{dx}{dt}=
    \left\{
  \begin{aligned}
        &a_1x\left(1-\frac{x}{K}\right),&  t\in \big[kT,kT +T_{1}\big),\ x\in \mathbb R_0^+,\\
        &-a_2 x,&  t\in \big[kT+ T_{1},(k+1)T\big),\ x\in \mathbb R_0^+,
    \end{aligned}
  \right.
\end{split}
\end{equation}
where  $a_1$ and $-a_2$ represent the reproductive rates in the corresponding seasons, and $K$ is the carrying capacity.
This model is applied in \cite{squirrel1} to study the population dynamics of squirrels (as can be seen in the works, e.g. \cite{squirrel1, squirrel2}, squirrel numbers typically achieve the peaks between March and September as a result of seasonal reproduction and decrease in winter due to food shortages and mortality). Furthermore, it is known that squirrels, as well as some other species, are also subject to the Allee effect, i.e., a negative effect on population growth when the species is underpopulated or overpopulated (see e.g. \cite{book, Allee} and the references therein). In \cite{Allee}, Li and Zhao modify the model \eqref{eq1} to a generalized form that considers the Allee effect in one of the seasons:
\begin{equation}\label{eq2}
\begin{split}
    \frac{dx}{dt}=
    \left\{
  \begin{aligned}
        &a_1x\left(1-\frac{x}{K_1}\right),&  t\in \big[kT,kT+ T_{1}\big),\ x\in \mathbb R_0^+,\\
        &a_2 x\left(1-\frac{x}{K_{21}}\right)\left(\frac{x}{K_{22}}-1\right),&  t\in \big[kT+ T_{1},(k+1)T\big),\ x\in \mathbb R_0^+,
    \end{aligned}
  \right.
\end{split}
\end{equation}
where $a_1,a_2,K_1>0$ and $K_{21}>K_{22}>0$. The authors give a criterion for the uniqueness and stability of the positive periodic solution of the model.

The second series of significant research is led by Yu et al. on the sterilization of wild mosquitoes using the factory-reared adult ones with artificial triple-{\em Wolbachia} infections \cite{YULI, 2, 3, 4, 5, BOYU, 1, 6, 7, 8}.
The population suppression model for the wild mosquitoes, as formulated and systematically analyzed by Yu, Li and Zheng in their series of works \cite{YULI,5,BOYU}, is given by
\begin{equation}\label{mosquito}
    \frac{dw}{dt}=\Big(\frac{aw}{w+h}-\big(\mu+\xi(w+h)\big)\Big)w,\ \ \ t,\omega\in\mathbb R_0^+,
\end{equation}
where $w(t)$ and $h=h(t)$ represent the numbers of wild mosquitoes and sexually active sterile mosquitoes (the factory-reared ones with {\em Wolbachia} infections) at time $t$, respectively;
the parameters $a$, $\mu$ and $\xi$ are the birth rate, the density-independent death rate and the density-dependent death rate of wild mosquitoes, respectively. An additional assumption $a>\mu$ is also imposed to ensure that, the wild mosquito population stabilizes at $A:=\frac{a-\mu}{\xi}$ in the absence of sterile mosquitoes.

The mechanism for sterilizing wild mosquitoes, which was reported in \cite{1}, is to release the sexually active sterile ones to mate with.
Currently, due to the limited production capacity of the mosquito factory, the release of the sterile mosquitoes is more appropriately implemented in periodic and impulsive strategies, under which the function $h(t)$ becomes a periodic step function and therefore the model \eqref{mosquito} has the form \eqref{equation0}.
There are three key design parameters for such strategies (see e.g. \cite{YULI, BOYU, 5}):
the period waiting time $T$ between two consecutive releases, the sexual lifespan $\overline{T}$ of the sterile mosquitoes, and the constant amount $c$ of each release.
In \cite{YULI,5} Yu and Li consider model \eqref{mosquito} under the strategy of $T>\overline{T}$, where $h(t)$ is $T$-periodic, satisfying
\begin{equation}\label{eq3}
\begin{split}
    h(t)=
    \left\{
  \begin{aligned}
        &c, \quad&t\in\big[kT,kT+\overline{T}\big),\\
        &0,\quad&t\in\big[kT+\overline{T},(k+1)T\big).
    \end{aligned}
   \right.
\end{split}
\end{equation}
Later, the model under the opposite strategy $T<\overline{T}$ is studied by Zheng et al. \cite{BOYU}, where
\begin{equation}\label{eq4}
\begin{split}
    h(t)=
    \left\{
  \begin{aligned}
        &(p+1)c, \quad&t\in\big[kT,kT+q\big),\\
        &pc,\quad&t\in\big[kT+q,(k+1)T\big),
    \end{aligned}
  \right.
\end{split}
\end{equation}
with $p=[\overline{T}/T]$ and $q=\overline T-pT$.
By virtue of these above works, the dynamical behavior of wild mosquitoes under the mainstream release strategies is almost completely obtained. For details readers are referred to the summary in \cite{BOYU} (see also Application 3 in the present paper). A comparison of the effects between different release strategies based on this dynamical description is also provided in the same article.


Another representative series of studies concerns the issue of open-closed seasonal harvesting in the context of fisheries and wildlife management policies (i.e., the resource exploitation is permitted during the open season but forbidden in the closed season),
see e.g. \cite{XIAO,xu,harvest2,Han,CW} and the references therein. In \cite{xu} Xu et al. theoretically investigated for the first time the logistic model with seasonal harvesting characterized by different periodic piecewise functions. Motivated by this work and the fixed quota policy (fishing or hunting licenses), Xiao \cite{XIAO} in 2016 formulated a single-species mathematical model with seasonal constant-yield harvesting
\begin{equation}\label{xiaodongmei}
\begin{split}
    \frac{dx}{dt}=
    \left\{
  \begin{aligned}
        &g(x),&  t\in \big[kT,kT+ T_{1}\big),\ x\in \mathbb R_0^+,\\
        &g(x)-h,&  t\in \big[kT+ T_{1},(k+1)T\big),\ x\in \mathbb R_0^+,\\
  \end{aligned}
    \right.
\end{split}
\end{equation}
where $h>0$, $T>T_1>0$ and the growth rate $g(x)$ satisfies the following hypothesis:
\begin{itemize}
    \item[(H)] $g\in C^{1}(\mathbb{R})$ and there exist positive numbers $k_{0}$ and $K$, $0<k_{0}<K$ such that $g(0)=g(K)=g^{'}(k_{0})=0$, $g^{'}(x)>0$ for $x\in[0,k_{0})$ and $g^{'}(x)<0$ for $x\in(k_{0},+\infty)$.
\end{itemize}
When $g=x(1-x)$ (a typical case of (H)), the author showed that there exists a unique threshold $h_{MSY}$, the {\em maximum sustainable yield}, such that model \eqref{xiaodongmei} undergoes a saddle-node bifurcation of the periodic solution as $h$ passes through $h_{MSY}$.
Later Han et al. \cite{Han} obtained the same conclusion under the hypotheses of (H), $g\in C^2(\mathbb R)$ and $g''<0$. The dynamics of a logistic model with seasonal Michaelis-Menten type harvesting is also recently studied by Feng et al. in \cite{harvest2}.

For more relevant works, see e.g. \cite{harvest,harvest3,BB} and the references therein.
\vskip0.2cm

Let us come back to the general equation \eqref{equation0}. Stimulated by the backgrounds mentioned above, we are concerned with the theoretical study of the dynamics of the equation.
The consideration will be restricted to the equation in one period $T$ (that is, the equation with $k=0$) and its periodic solutions.
Here we briefly recall several notions. For each $i=1,\ldots,n$, equation \eqref{equation0} in the strip $[T_{i-1},T_{i}]\times\mathbb R$ has a unique solution $x=\varphi_i(t;t_{0},x_{0})$, which satisfies the initial condition $\varphi_i(t_{0};t_{0},x_{0})=x_{0}$ and is determined by the sub-equation $\frac{dx}{dt}=f_i(x)$. Taking the composition of functions into account, for each initial value $(t_0,x_0)\in[0,T]\times\mathbb R$, the maximal solution $x=x(t;t_{0},x_{0})$ of equation \eqref{equation0}, is unique and formed recursively by means of
\begin{align*}
  x(t;t_{0},x_{0})=
  \left\{
  \begin{aligned}
  &\varphi_i(t;t_{0},x_{0}), &\text{when}\  t_0, t\in[T_{i-1}, T_{i}],\\
  &\varphi_i(t;T_{i-1},x(T_{i-1};t_{0},x_{0})), &\text{when}\ t_0< T_{i-1}\ \text{and}\ t\in[T_{i-1}, T_{i}],\\
  &\varphi_i(t;T_{i},x(T_{i};t_{0},x_{0})), &\text{when}\ t_0>T_{i}\ \text{and}\ t\in[T_{i-1}, T_{i}].
  \end{aligned}
  \right.
\end{align*}
We say that $x=x(t;0,x_0)$ is periodic if it is well-defined on $[0,T]$ and satisfies $x(T;0,x_0)=x_0$. Also a periodic solution is called a {\em limit cycle} if it is isolated in the set of periodic solutions.

To our knowledge, there are several analytical approaches to the periodic solutions of equation \eqref{equation0} in the previous works. They are mainly based on either the specific expressions of $\varphi_i$'s, each obtained from the integrability of the corresponding sub-equation (see e.g. \cite{Allee,squirrel1,YULI,5,BOYU,harvest2}), or, the derivative formulas in integral form of the Poincar\'{e} map $P(x_0):=x(T;0,x_0)$ (see e.g. \cite{XIAO,Han}), which are derived from the theory of the smooth non-autonomous differential equations \cite{Lloyd}. Our purpose in this paper, is to further supplement some specialized qualitative tools for analyzing the periodic solutions of equation \eqref{equation0}, especially when the number of sub-equations $n=2$, enabling more sufficient and effective utilization of the particularities of the equation.

To achieve this, we apply Lemma \ref{lemma1} given in Section 2, which is a trivial adaptation of \cite[Lemma 3.9] {HX}, to normalize the consideration to the study of the following case:
\begin{equation}\label{equation1}
\begin{split}
    \frac{dx}{dt}=f(t,x)=
    \left\{
  \begin{aligned}
        &f_{1}(x),&  t\in \Big[0,\, \frac{1}{n}\Big),\ x\in I,\\
        &f_{2}(x),&  t\in \Big[\frac{1}{n},\, \frac{2}{n}\Big),\ x\in I,\\
        &\quad\cdots  \quad& \quad\cdots \quad\\
        &f_{n}(x),&  t\in \Big[\frac{n-1}{n},\, 1\Big],\ x\in I,
    \end{aligned}
  \right.
\end{split}
\end{equation}
i.e., equation \eqref{equation0} with $k=0$, $T=1$ and $T_i=\frac{i}{n}$, $i=0,1,\ldots,n$. We recall the assumption that $f_i$'s are $C^1$-differentiable on $I$ and only have a finite number of zeros. Then the discussion above naturally ensures the $C^1$-differentiability for the Poincar\'{e} map $P(x_0)=x(1;0,x_0)$ of equation \eqref{equation1} on its domain, denoted by $\mathcal D\subseteq I$.
Our first main result, further presents the $C^2$-differentiability and two derivative formulas for $P$ on $\mathcal D$ with an exception of a finite set.

\begin{theorem}\label{th1'}
    The Poincar\'{e} map $P(x_0)=x(1;0,x_0)$ of equation \eqref{equation1}, is $C^2$-differentiable on $V:=\big\{\rho\in \mathcal D\big|f_i\big(x(\frac{i-1}{n};0,\rho)\big)\neq0,i=1,2,\cdots,n\big\}$.
    Furthermore, for $x_0\in V$ and $i=1,2,\ldots,n$, let $x_i:=x(\frac{i}{n};0,x_0)$. Then
\begin{align}\label{12}
\begin{split}
    &P'(x_0)=\prod_{i=1}^{n}\frac{f_i(x_i)}{f_i(x_{i-1})},\\
    &P''(x_0)=P'(x_0)\sum_{i=1}^{n}\left(\frac{f^{'}_{i}(x_{i})-f^{'}_{i}(x_{i-1})}{f_{i}(x_{i-1})}\cdot \prod_{j=1}^{i-1}\frac{f_{j}(x_{j})}{f_{j}(x_{j-1})}\right).
\end{split}
\end{align}
\end{theorem}

We remark that although the solutions $x(t;0,x_0)$ of the equation with $x_0\in\mathcal D\backslash V$ are not involved in Theorem \ref{th1'}, they are relatively simple (such as the constant solutions $x(t;0,x_0)\equiv x_0$) and can be located by the finitely many zeros of $f_i$'s.

Under suitable differentiability assumption, we also obtain the expression of $ P'''$ on $V$. For the sake of brevity, it is stated in Section 3.

\vskip 0.2cm

In what follows we focus on a more concrete case of equation \eqref{equation1}, that is, the equation with $n=2$:
\begin{equation}\label{equation2}
\begin{split}
    \frac{dx}{dt}=
    \left\{
  \begin{aligned}
        &f_{1}(x),& \quad t\in \Big[0,\frac{1}{2}\Big),\ x\in I,\\
        &f_{2}(x),& \quad t\in \Big[\frac{1}{2},1\Big],\ x\in I.
    \end{aligned}
  \right.
\end{split}
\end{equation}
For this equation, we immediately obtain the following result by Theorem \ref{th1'}.
\begin{theorem}\label{coro1}
    If $x(t;0,x_{0})$ is a non-constant periodic solution of equation \eqref{equation2}, then the Poincar\'{e} map $P(x_0)=x(1;0,x_{0})$ is $C^2$-differentiable at $x_0$ and
    \begin{equation*}
        P^{'}(x_{0})=1-\frac{1}{f_{1}(x_{0})f_{2}(x_{1})}\,
        \begin{vmatrix}
            f_{1}(x_{0})&f_{2}(x_{0})\\
            f_{1}(x_{1})&f_{2}(x_{1})
        \end{vmatrix},
    \end{equation*}
    where $x_1=x(\frac{1}{2};0,x_0)$.
    Furthermore, if $P^{'}(x_{0})=1$, then
\begin{equation*}
    \begin{split}
        P^{''}(x_{0})&=\frac{1}{f_{1}(x_{0})f_{2}(x_{1})}
        \begin{vmatrix}
                f_{1}^{'}(x_{1})-f_{1}^{'}(x_{0}) & f_{2}^{'}(x_{1})-f_{2}^{'}(x_{0})\\
                f_{1}(x_{1}) & f_{2}(x_{1})
        \end{vmatrix}\\
        &=\frac{1}{f_{1}(x_{0})f_{2}(x_{0})}
        \begin{vmatrix}
                f_{1}^{'}(x_{1})-f_{1}^{'}(x_{0}) & f_{2}^{'}(x_{1})-f_{2}^{'}(x_{0})\\
                f_{1}(x_{0}) & f_{2}(x_{0})
            \end{vmatrix}.
    \end{split}
\end{equation*}
\end{theorem}
Next, using a combination of the above results and additional analysis, we can characterize the regions in which equation \eqref{equation2} may admit periodic solutions, in terms of the functions $f_1, f_2$ and $f_1+f_2$.

\begin{theorem}\label{th2}
    Suppose that for equation \eqref{equation2}, the zeros (if any) of $f_1+f_2$ on $I$ are all isolated. Let $I_{1},I_{2},\cdots,I_{m}$ be the non-intersecting open intervals given by all the connected components of $I\backslash\{x|f_1(x)f_2(x)=0\}$, where $m=1+{\rm Card}\{x|f_1(x)f_2(x)=0\}$. The following statements hold.
\begin{itemize}
  \item[(i)] Suppose that $x=x(t;0,x_0)$ is a periodic solution of equation \eqref{equation2}.
  \begin{itemize}
  \item[(i.1)] If $x(t;0,x_0)\equiv x_0$, then $x_0$ is a common zero of $f_{1}$ and $f_{2}$  (i.e., $f_1(x_0)=f_2(x_0)=0$).
  \item[(i.2)] If $x(t;0,x_0)\not\equiv x_0$, then the solution is located in the region $\big\{(t,x)\big|t\in [0,1], x\in \bigcup_{i=1}^{m}I_{i}\big\}$ (i.e., $[0,1]\times \bigcup_{i=1}^{m}I_{i}$). Furthermore, $f_1+f_2$ is sign-changing on the set $\{x(t;0,x_{0})|t\in [0,1]\}$.
  \end{itemize}
  \item[(ii)] Suppose that $f_{1}+f_{2}$ changes sign exactly once on an open interval $E\subseteq\bigcup^{m}_{i=1} I_{i}$. Then equation \eqref{equation2} has at most one periodic solution in the region $\{(t,x)|t\in [0,1], x\in E\}$ (i.e., $[0,1]\times E$). Furthermore, if such a periodic solution exists and the number of zeros of $f_{1}+f_{2}$ on $E$ is exactly one, then the solution is hyperbolic.
\end{itemize}

\end{theorem}

Theorems \ref{coro1} and \ref{th2} in fact provide a preliminary procedure to explore the non-existence, the existence and the number of limit cycles of equation \eqref{equation2}. The procedure is outlined below in two steps and will be used repeatedly later on in this paper:

\vskip0.1cm
\begin{enumerate}
[\textbf{Step 1:}] \item Determine all the common zeros of $f_{1}$ and $f_{2}$ (the constant limit cycles of equation \eqref{equation2}). Also, divide $I$ into the open intervals $I_{i}$'s by the zeros of $f_{1}$ and $f_{2}$.
\end{enumerate}
\begin{enumerate}
[\textbf{Step 2:}] \item Determine the number of sign changes $s_{I_i}$ of $f_{1}+f_{2}$ on each interval $I_{i}$.
\begin{itemize}
\item[$\bullet$] If $s_{I_i}=0$, no limit cycles of equation \eqref{equation2} exist in the region $[0,1]\times I_{i}$ (by Theorem \ref{th2}).

\item[$\bullet$] If $s_{I_i}=1$, at most one limit cycle of equation \eqref{equation2} exists in the region $[0,1]\times I_{i}$ (by Theorem \ref{th2}).

\item[$\bullet$] If $s_{I_i}\geq2$, analyze the stability and the multiplicity of the limit cycles that could exist in the region $[0,1]\times I_{i}$ (using the derivative formulas of the Poincar\'{e} map from Theorem \ref{coro1} and the ones previously established). To estimate the number of limit cycles of the equation in this case, further analysis such as monotonicity and bifurcation analysis could be applied.
\end{itemize}
\end{enumerate}
\vskip0.1cm


Here we illustrate three applications that show the effectiveness of our results in studying the limit cycles of some concrete population models with time-varying factors. All of these models are derived from the real-world problems introduced in the previous backgrounds. 

\subsection*{Application 1}
The first application is on the global dynamics of the seasonal harvesting model \eqref{xiaodongmei} under the hypothesis (H) proposed by  Xiao \cite{XIAO}.
We recall that for the case where the growth rate $g=x(1-x)$, Xiao proved in the article \cite[Theorem 3.3] {XIAO} that model \eqref{xiaodongmei} has at most two positive limit cycles and exhibits a saddle-node bifurcation with respect to the parameter $h$. Han et al. \cite[Theorem 3.6] {Han} obtained the same conclusion under the hypothesis (H) with the addition of $g\in C^2(\mathbb R)$ and $g''<0$.

Here we go further and give the global dynamics of model \eqref{xiaodongmei}, which depends only on hypothesis (H).
We believe that this result could also be practical because there are indeed some models where the growth rate functions satisfy (H) and are $C^2$-differentiable, but with the second derivatives changing signs on $[0,\infty)$ (see e.g. the population models with the Von Bertalanffy's growth rate, the hyper-logistic growth rate and the generic growth rate in \cite{logistic,TBKP,B}).




\begin{theorem}\label{application1}
    Under hypothesis $\mathrm{(H)}$, equation \eqref{xiaodongmei} has at most two limit cycles, taking into account their multiplicities. Furthermore, there exists a threshold $h_{MSY}\in \big(h^*, \frac{T}{T-T_1}h^*\big)$, where $h^*=g(k_{0})$, such that
    \begin{itemize}
        \item[(i)] Equation \eqref{xiaodongmei} has exactly two limit cycles if $0<h< h_{MSY}$, where the lower one is unstable and the upper one is stable;
        \item[(ii)] Equation \eqref{xiaodongmei} has a unique semi-stable limit cycle if $h=h_{MSY}$;
        \item[(iii)] Equation \eqref{xiaodongmei} has no limit cycles if $h> h_{MSY}$.
    \end{itemize}
\end{theorem}
 The proof of Theorem \ref{application1} is given in Section 4.


\subsection*{Application 2}
The second application is on the model defined by the equation
\begin{align}\label{main equation}
\begin{split}
    \frac{dx}{dt}
    =
    \left\{
  \begin{aligned}
         &a_{1}x^3+b_{1}x^2+c_{1}x,&  t\in \big[kT,kT+ T_{1}\big),\ x\in\mathbb R,\\
         &a_{2}x^3+b_{2}x^2+c_{2}x,&  t\in \big[kT+ T_{1},(k+1)T\big),\ x\in\mathbb R,
    \end{aligned}
  \right.
\end{split}
\end{align}
where $a_i, b_i, c_i\in\mathbb R$, $i=1,2$, and $T>T_1>0$.
  It is clear that equation \eqref{main equation} has several seasonal population models as special cases, such as the mentioned model \eqref{eq1}, model \eqref{eq2}, and the model for the species exhibiting the Allee effect in each season
\begin{align*}
    &\frac{dx}{dt}=
    \left\{
    \begin{aligned}
        &a_{1}x\left(1-\frac{x}{K_{11}}\right)\left(\frac{x}{K_{12}}-1\right),&  t\in \big[kT,kT+ T_{1}\big),\ x\in\mathbb R_0^+,\\
        &a_{2}x\left(1-\frac{x}{K_{21}}\right)\left(\frac{x}{K_{22}}-1\right),&  t\in \big[kT+ T_{1},(k+1)T\big),\ x\in\mathbb R_0^+,
    \end{aligned}
    \right.
\end{align*}
where $a_{i}>0$ and $K_{i1}>K_{i2}>0$, $i=1,2$. Moreover, the mosquito population suppression model \eqref{mosquito}, with the release strategy that satisfies \eqref{eq3}, is written in the form \eqref{main equation} under a simple Cherkas' transformation $x=\frac{\omega}{\omega+c}$ (see \cite{Cherkas} and Remark \ref{rem}).

We remark that equation \eqref{main equation} is of the form $\frac{dx}{dt}=a(t)x^3+b(t)x^{2}+c(t)x$,
which is called {\em Abel equation} of the first kind (usually, Abel equation for short). There has been a long list of works (over three hundred) applying and studying Abel equations in the literature. For more applications of the equations to real problems, the reader is referred to the reductions of the SIR model \cite{Har}, the reaction-diffusion model describing the evolution of glioblastomas \cite{Har2}, the model arising in a tracking control problem \cite{control}, etc.
In the area of qualitative theory, the Abel equations also play an important role in the study of the Hilbert's 16th problem for some planar differential systems, see e.g. \cite{Lloyd,Lloyd2,Llibre,openproblem} and the references therein. A natural problem in these contexts, sometimes referred to as the Pugh's problem, asks for the estimates of the number of limit cycles of Abel equations (see e.g. \cite{Lins,Panov1,Llibre,Ben,Pliss}).

The complexity of the Pugh's problem was first known in the significant article \cite{Lins} by Lins-Neto, where the author showed that the general Abel equations have no uniform upper bound for the number of limit cycles without additional restrictions. So far, the criteria of bounding the number of limit cycles of the equations, mainly require the fixed sign hypotheses for the linear combinations of the coefficients $a(t),b(t)$ and $c(t)$, see e.g. \cite{Lins,Llibre,openproblem,HX,HL}.
One of the best known, presented in \cite{Lins,Llibre}, states that the smooth Abel equations have at most three limit cycles when $a(t)$ does not change sign. This criterion actually applies to equation \eqref{main equation} as well, where the hypothesis is correspondingly written as $a_1a_2>0$, see \cite{HX}. To the best of our knowledge, this is currently the only estimate for equation \eqref{main equation}.

In view of the previous works for the Pugh's problem, one notes that the established tools depend on the solutions of the Abel equations over the whole period. This results in the main challenge of the analysis because such solutions are complicated and typically have unknown expressions.
As will be seen in Section 4, our specialized approach to equation \eqref{main equation} utilizes the relationship between the values of the solutions at two points $t=0$ and $t=T_1$, which is a bit more feasible and enables us to provide an answer to the Pugh's problem for the equation:
\begin{theorem}\label{main theorem}
    Equation \eqref{main equation} has at most three limit cycles, taking into account their multiplicities. This upper bound is sharp.
\end{theorem}

\subsection*{Application 3}
The third application is on the mosquito population suppression model \eqref{mosquito}. As introduced above, the realistic release strategies of the sterile mosquitoes, are designed according to the periodic waiting time $T$ between the releases, the sexual lifespan $\overline T$ of the sterile mosquitoes, and the release amount $c$. We adopt the classification from the articles \cite{YULI,5,BOYU}, considering the model under the following two strategies:

\subsubsection*{Strategy $T>\overline T$}
In this case the function $h(t)$ in model \eqref{mosquito} satisfies \eqref{eq3}. Yu and Li \cite{YULI,5} successively investigated the existence and stability of limit cycles of such model.
They introduced two release amount thresholds $g^{*}$ and $c^{*}$ with $g^{*}<c^{*}$, and a waiting period threshold $T^{*}>\overline T$, see \cite{YULI} or \eqref{thresholds1} in Section 4 for the explicit expressions. By means of the integrability techniques, they obtained the following result.

\begin{proposition}[\cite{YULI, 5}]\label{subapplication2.1}
    Under the assumption that $T>\overline T$ (i.e., \eqref{eq3} is satisfied),
    equation \eqref{mosquito} has at most three limit cycles. Furthermore, the global dynamics of the equation with respect to the parameters $T$ and $c$ is characterized in Table \ref{table1}, where $E_{0}$ represents the trivial constant solution $\omega=0$, and the notations ``\textnormal{NC}'', ``\textnormal{LAS}'', ``\textnormal{GAS}'' and ``\textnormal{US}'' represent ``non-constant'', ``locally asymptotically stable'', ``globally asymptotically stable'' and ``unstable'', respectively.
    \vspace{-0.2cm}
    \begin{table}[!ht]
        \center
        \begin{threeparttable}
        \begin{tabular}{|c|c|c|c|}
            \hline
             \diagbox[innerwidth=2cm]{$T$}{$c$} &$0<c\leq g^{*}$&$g^{*}<c<c^{*}$&$c\geq c^{*}$\\
            \hline
           $\overline{T}<T<T^{*}$&\makecell[c]{\rm{Exact two NC limit }\\\rm{cycles +$E_{0}$ LAS}}&\makecell[c]{\rm{At most two NC limit }\\ \rm{cycles + $E_{0}$ LAS}}&\multicolumn{1}{c|}{\multirow{2}{*}{\rm{$E_{0}$ GAS}}} \\
           \cline{1-3}
           $T=T^{*}$& \multicolumn{2}{c|}{\multirow{2}{*}{\makecell[c]{\rm{A unique GAS NC limit cycle}\\ $\iff$ \rm{$E_{0}$ US}}}}&\\
           \cline{1-1} \cline{4-4}
           $T>T^{*}$& \multicolumn{3}{c|}{} \\
           \hline
        \end{tabular}
        \end{threeparttable}
        \caption{Main conclusions in \cite{YULI,5}.}\label{table1}
        \vspace{-0.8cm}
    \end{table}
\end{proposition}
Here we reobtain this result by using our approach. The proof will be more short and simple compared to the ones in \cite{YULI,5}.
See Section 4 for details.

\subsubsection*{Strategy $T<\overline T$}
In this case the function $h(t)$ in model \eqref{mosquito} satisfies \eqref{eq4}. Similar to the studies in \cite{YULI,5}, Zheng et al. \cite{BOYU} introduced four thresholds $g_{1}^{*}$, $g_{2}^{*}$, $T^{**}$ and $c^{**}$, where the explicit expressions of the first two are also shown in \eqref{thresholds2} in Section 4. The authors got the following global dynamics for the model.
\begin{proposition}[\cite{BOYU}]\label{BOYU}
     Under the assumption that $T<\overline T$ (i.e., \eqref{eq4} is satisfied), the global dynamics of equation \eqref{mosquito} with respect to the parameters $T$ and $c$ is characterized in Table \ref{table2}, where the notations ``$E_0$'', ``\textnormal{NC}'', ``\textnormal{LAS}'', ``\textnormal{GAS}'' and ``\textnormal{US}'' follow the same as in Proposition \ref{subapplication2.1}, and ``?'' represents the unknown result.
     \vspace{-0.2cm}
    \begin{table}[!ht]
        \center
        \label{Table2}
        \begin{threeparttable}
        \begin{tabular}{|c|c|c|c|c|}
            \hline
             \diagbox[innerwidth=2cm]{$T$}{$c$} &$0<c\leq g_{1}^{*}$&$g_{1}^{*}<c\leq c^{**}$&$c^{**}<c<g_{2}^{*}$&$c\geq g_{2}^{*}$\\
            \hline
           $0<T<T^{**}$& \multicolumn{1}{c|}{\multirow{3}{*}{\makecell[c]{\rm{Exact two NC}\\ \rm{limit cycles}\\ \rm{+$E_{0}$ LAS}}}}  &\multicolumn{2}{c|}{\rm{$E_{0}$ GAS}}&\multicolumn{1}{c|}{\multirow{3}{*}{\rm{$E_{0}$ GAS}}} \\
           \cline{1-1}\cline{3-4}
           $T=T^{**}$& &\textcolor{red}{\rm{?}}&\textcolor{red}{\rm{?}}& \\
           \cline{1-1}\cline{3-4}
           $T^{**}<T<\overline{T}$& &\textcolor{red}{\rm{?}}+ \rm{$E_{0}$ LAS}&\makecell[c]{\rm{At most two NC limit}\\ \rm{cycles+ $E_{0}$ LAS}}& \\
           \hline
        \end{tabular}
        \end{threeparttable}
        \caption{Main conclusions in \cite{BOYU}.}\label{table2}
        \vspace{-0.8cm}
    \end{table}
\end{proposition}

As presented in Table \ref{table2}, the characterization of the dynamics of equation \eqref{mosquito}, remains incomplete when $c\in(g_{1}^{*},g_{2}^{*})$.
Here, we address this and improve the result of Proposition \ref{BOYU} by introducing a new threshold $T^{***}$, defined in \eqref{threshold3}. Our result is stated as below:

\begin{theorem}\label{subapplication2.2}
    Under the assumption that $T>\overline T$ (i.e., \eqref{eq4} is satisfied), the global dynamics of equation \eqref{mosquito} with respect to the parameters $T$ and $c$ is characterized in Table \ref{table3}, where the notations ``$E_0$'', ``\textnormal{NC}'', ``\textnormal{LAS}'', ``\textnormal{GAS}'' and ``\textnormal{US}'' follow the same as in Proposition \ref{subapplication2.1}.
    \vspace{-0.2cm}
    \begin{table}[!ht]
        \center
        \label{Table2}
        \begin{threeparttable}
        \begin{tabular}{|c|c|c|c|}
            \hline
             \diagbox[innerwidth=2cm]{$T$}{$c$} &$0<c\leq g_{1}^{*}$& $g_{1}^{*}<c<g_{2}^{*}$&$c\geq g_{2}^{*}$\\
            \hline
           $0<T<T^{***}$& \multicolumn{1}{c|}{\multirow{3}{*}{\makecell[c]{\rm{Exact two NC}\\ \rm{limit cycles}\\ \rm{+$E_{0}$ LAS}}}}  &\multicolumn{1}{c|}{\multirow{2}{*}{\textcolor{blue}{\rm{$E_{0}$ GAS}}}}&\multicolumn{1}{c|}{\multirow{3}{*}{\rm{$E_{0}$ GAS}}} \\
           \cline{1-1}
           $T=T^{***}$& & & \\
           \cline{1-1}\cline{3-3}
           $T^{***}<T<\overline{T}$& & \makecell[c]{\textcolor{blue}{\rm{At most two NC limit cycles}}\\\textcolor{blue}{\rm{+ $E_{0}$ LAS}}}& \\
           \hline
        \end{tabular}
        \end{threeparttable}
        \caption{Completion for the conclusions in \cite{BOYU}.}\label{table3}
        \vspace{-0.8cm}
    \end{table}
\end{theorem}

The organization of this paper is as follows.
Section 2 provides some auxiliary results, adapted from the previously known ones, which will be applied in Section 4 together with our results. Section 3 is devoted to proving Theorems \ref{th1'}, \ref{coro1} and \ref{th2}.
The proofs of the results (Theorems \ref{application1}, \ref{main theorem}, \ref{subapplication2.2} and Proposition \ref{subapplication2.1}) in the three applications using our new approach are presented in Section 4.
The last part is appendix.

\section{Preliminaries}
There are three parts of the auxiliary results in this section. They are presented one by one in the following.

\subsection{``Normal'' form of equation \eqref{equation0}}
In the first subsection, we follow the idea of \cite[Lemma 3.9] {HX} and give a simple normalization of equation \eqref{equation0} to the one with $T_i=\frac{i}{n}$, $i=0,\ldots,n$, that is, the one of the form \eqref{equation1}. More concretely, consider the equation
\begin{equation}\label{00}
    \frac{dx}{dt}=
\left\{
    \begin{aligned}
    &n(T_{1}-T_{0})f_{1}(x),&  t\in \Big[k,k+\frac{1}{n}\Big),\ x\in I,\\
    &n(T_{2}-T_{1})f_{2}(x),&  t\in \Big[k+\frac{1}{n},k+\frac{2}{n}\Big),\ x\in I,\\
    &\quad\cdots  \quad& \quad\cdots \quad\\
    &n(T_{n}-T_{n-1})f_{n}(x),&  t\in \Big[k+\frac{n-1}{n},k+1\Big),\ x\in I,
\end{aligned}
\right.
\end{equation}
where $T_{0},\ldots,T_n$, $f_{1},\ldots,f_n$ and $k$ are defined as in \eqref{equation0}.
Let $x(t;t_0,x_0)$ and $\overline x(t;t_0,x_0)$ be the solutions of equation \eqref{equation0} and equation \eqref{00} respectively, with the initial conditions $x(t_0;t_0,x_0)=\overline x(t_0;t_0,x_0)=x_0$. We have the following lemma:


\begin{lemma}\label{lemma1}
The Poincar\'{e} maps $P(x_0)=x(T;0,x_0)$ of equation \eqref{equation0} and $\overline P(x_0)=\overline x(1;0,x_0)$ of equation \eqref{00}, satisfy $P=\overline P$.
\end{lemma}

\begin{proof}
    Let $P_i$ and $\overline P_i$ be the Poincar\'{e} maps of equation \eqref{equation0} from $t=T_{i-1}$ to $t=T_i$ and equation \eqref{00} from $t=\frac{i-1}{n}$ to $t=\frac{i}{n}$, respectively, i.e.,
    $P_i(x_0)=x(T_i;T_{i-1},x_0)$ and $\overline P_i(x_0)=\overline x\big(\frac{i}{n};\frac{i-1}{n},x_0\big)$, where $i=1,\ldots,n$.
    By assumption, we have the decompositions $P=P_n\circ\cdots\circ P_1$ and $\overline P=\overline P_n\circ\cdots\circ \overline P_1$. Then, the conclusion will follow once  $P_i=\overline P_i$ for each $i$.

    In fact, denote by $\mu_i(t)=n(T_{i}-T_{i-1})(t-\frac{i-1}{n})+T_{i-1}$, where $t\in\big[\frac{i-1}{n},\frac{i}{n}\big)$. One can readily see that $\mu_i(t)\in[T_{i-1},T_i)$ with $\mu_i\big(\frac{i-1}{n}\big)=T_{i-1}$ and $\mu_i\big(\frac{i}{n}\big)=T_{i}$. Thus,
     \begin{align*}
       \frac{d x(\mu_i(t);T_{i-1},x_0)}{dt}=\frac{dx}{d\mu_i}\frac{d\mu_i}{dt}=n(T_{i}-T_{i-1}) f_i\big(x(\mu_i(t);T_{i-1},x_0)\big), 
     \end{align*}
    which indicates that $x(\mu_i(t);T_{i-1},x_0)$ is a solution of the $i$-th sub-equation of \eqref{00}. Taking into account the fact that $x\big(\mu_i(\frac{i-1}{n});T_{i-1},x_0\big)=x(T_{i-1};T_{i-1},x_0)=x_0$, we obtain $x\big(\mu_i(t);T_{i-1},x_0\big)=\overline x\big(t;\frac{i-1}{n},x_0\big)$. Accordingly, $P_i=\overline P_i$ holds in its domain, $i=1,\ldots,n$. The proof of the lemma is finished.
    \end{proof}
In view of the arbitrariness of the functions $f_1,\ldots,f_n$, Lemma \ref{lemma1} implies that the dynamical problem for the general equation \eqref{equation0} can be reduced to that for the particualr case \eqref{equation1}. The latter one will therefore be our main consideration below (as stated in Section 1), unless explicitly stated otherwise.

\subsection{Rotated differential equations}
The theory of rotated differential equations (monotonic family of differential equations) is a direct adaptation of the classical theory of rotated vector fields introduced by Duff \cite{rotated1}. It provides powerful tools for understanding the evolution of limit cycles as the parameter of the differential equations varies, see e.g. \cite{rotated2,rotated3}.
Han et al. \cite{Han} in 2018 systematically generalized this theory to the one-dimensional piecewise smooth periodic differential equations.
This subsection recalls the related concept and several results of such kind of equations. For further information the readers are referred to the work \cite{Han}.

\begin{definition}[\cite{Han}]\label{definition of rotated}
    Consider a family of differential equations
    \begin{equation}\label{nah}
        \frac{dx}{dt}=f(t,x,\alpha),\indent\ t\in \mathbb{R},\ x\in I,\ \alpha\in J,
    \end{equation}
    where $I, J$ are intervals in $\mathbb R$, and the function $f$ satisfies the following conditions:
\begin{itemize}
    \item There exists a constant $T$ such that $f(t+T,x,\alpha)=f(t,x,\alpha)$ for all $(t,x,\alpha)\in \mathbb{R}\times I\times J$, i.e., $f$ is $T$-periodic in variable $t$.
    \item
    There exist $n+1$ constants $T_{0},T_{1},\cdots,T_{n}$ with $0=T_{0}<T_{1}<\cdots<T_{n-1}<T_{n}=T$, and $n$ functions $f_{1},\cdots,f_{n}\in C^{1}(\mathbb R\times I\times J)$, such that $f|_{\mathcal I_j\times I\times J}=f_j|_{\mathcal I_j\times I\times J}$, where $\mathcal I_j=[T_{j-1},T_j)$ and $j=1,\cdots,n$.
\end{itemize}
Then, we say that \eqref{nah} defines a family of rotated equations (a monotonic family of equations) on $[0,T]\times I$ with respect to $\alpha$,
if $\frac{\partial f}{\partial \alpha}\big|_{[0,T]\times I\times J}\geq 0$ (or $\leq 0$) and does not identically vanish along any solution of \eqref{nah}.
\end{definition}

\begin{lemma}[\cite{Han}]\label{property of rotated}
Suppose that equation \eqref{nah} defines a family of rotated equations with respect to $\alpha$.
\begin{itemize}
\item[(i)] If $x=\varphi(t)$ is a stable (resp. unstable) limit cycle of equation \eqref{nah}$|_{\alpha=\alpha_0}$, then there exists $\delta>0$ such that for $\alpha\in J\cap(\alpha_0-\delta,\alpha_0+\delta)$, equation \eqref{nah} has a stable (resp. unstable) limit cycle $x= x(t;\alpha)$ with $x(t;\alpha_0)=\varphi(t)$. Furthermore, the monotonicity of $x= x(t;\alpha)$ with respect to $\alpha$ is characterized in Table \ref{Table4}.
\item[(ii)] If $x=\varphi(t)$ is a semi-stable limit cycle of equation \eqref{nah}$|_{\alpha=\alpha_0}$, then there exists $\delta>0$ such that for $\alpha\in J\cap(\alpha_0-\delta,\alpha_0+\delta)\backslash\{\alpha_0\}$, equation \eqref{nah} has either two limit cycles located on distinct sides of $x=\varphi(t)$, or no limit cycles located near $x=\varphi(t)$. Furthermore, the existence and non-existence of these limit cycles as $\alpha$ varies from $\alpha_0$ are characterized in Table \ref{Table4}.
\end{itemize}\vspace{-0.1cm}
\begin{table}[!ht]
    \centering
    \newcommand{\tabincell}[2]{\begin{tabular}{@{}#1@{}}#2\end{tabular}}
    \begin{threeparttable}
    \begin{tabular}{|c|c|c|c|c|}
        \hline
        \tabincell{c}{behavior of \\ limit cycle} &stable&\tabincell{c}{unstable} &\tabincell{c}{upper-stable\\ lower-unstable}&\tabincell{c}{upper-unstable\\ lower-stable}\\
        \hline
        \tabincell{c}{$\frac{\partial f}{\partial \alpha}\geq 0$}& \tabincell{c}{increasing\\ in $\alpha$} &\tabincell{c}{decreasing\\ in $\alpha$} &\tabincell{c}{splits as $\alpha>\alpha_0$;\\ disappears as $\alpha<\alpha_0$}&\tabincell{c}{disappears as $\alpha>\alpha_0$;\\ splits as $\alpha<\alpha_0$}\\
       \hline
       $\frac{\partial f}{\partial \alpha}\leq 0$ & \tabincell{c}{decreasing\\ in $\alpha$} &\tabincell{c}{increasing\\ in $\alpha$} &\tabincell{c}{disappears as $\alpha>\alpha_0$;\\ splits as $\alpha<\alpha_0$}&\tabincell{c}{splits as $\alpha>\alpha_0$;\\ disappears as $\alpha<\alpha_0$}\\
       \hline
    \end{tabular}
    \end{threeparttable}
    \caption{Behavior of limit cycle(s) of \eqref{nah} as $\alpha$ varies. Here ``split'' means that $x=\varphi(t)$ bifurcates into two limit cycles  located on distinct sides of it, and ``disappear'' means that there is no limit cycle located near $x=\varphi(t)$.}\label{Table4}\vspace{-0.5cm}
\end{table}
\end{lemma}


\subsection{Derivative formulas in integral form of Poincar\'{e} map}
We begin this subsection by presenting the first three derivative formulas of the Poincar\'{e} map of equation \eqref{equation1}, which are generalized from the classical results for smooth equations in \cite{Lloyd}. The first two formulas are already established in \cite{Han,XIAO} and the third one is supplemented here. These formulas then yield a criterion, similar to that in \cite{Lloyd}, for controlling the number of periodic solutions of equation \eqref{equation1}. A result applying them to the constant solutions of equation \eqref{equation2} is also given. We remark that although not the main tools used in this work, the results in this subsection will provide valuable assistance later in the argument of Applications 2 and 3.


\begin{lemma}\label{han}
    The Poincar\'{e} map $P(x_0)=x(1;0,x_{0})$ of equation \eqref{equation1}, satisfies the following statements on its domain $\mathcal D\subseteq I$.
\begin{itemize}
  \item[(i)]{\upshape(\cite{Han})} Under assumption of $f_i$ (i.e., $f_{i}\in C^{1}(I)$ for $i=1,2,\ldots,n$), $P\in C^1(\mathcal D)$ and
  \begin{equation}\label{han1}
        P^{'}(x_{0})
        =\exp\int_{0}^{1}\frac{\partial f}{\partial x}\big(t,x(t;0,x_{0})\big)dt.
    \end{equation}
  \item[(ii)]{\upshape(\cite{Han})} If additionally $f_{i}\in C^{2}(I)$ for $i=1,2,\ldots,n$, then $P\in C^2(\mathcal D)$ and
  \begin{equation}\label{han2}
        \begin{split}
            P^{''}(x_{0})
            =P^{'}(x_{0})\int_{0}^{1}\frac{\partial^2 f}{\partial x^2}\big(t,x(t;0,x_{0})\big)\frac{\partial x}{\partial x_{0}}(t;0,x_{0})dt,
        \end{split}
        \end{equation}
  with $\frac{\partial x}{\partial x_{0}}(t;0,x_{0})=\exp\int_{0}^{t}\frac{\partial f}{\partial x}\big(s,x(s;0,x_{0})\big)ds$.
  \item[(iii)]  If additionally $f_{i}\in C^{3}(I)$ for $i=1,2,\ldots,n$, then $P\in C^3(\mathcal D)$ and
        \begin{equation}\label{han3}
        \begin{aligned}
            P^{'''}(x_{0})
            &=P^{'}(x_{0})\left[\frac{3}{2}\bigg(\frac{P^{''}(x_{0})}{P^{'}(x_{0})}\bigg)^{2}+\int_{0}^{1}\frac{\partial^3 f}{\partial x^3}\big(t,x(t;0,x_{0})\big)\bigg(\frac{\partial x}{\partial x_{0}}(t;0,x_{0})\bigg)^{2}dt\right].
        \end{aligned}
    \end{equation}
\end{itemize}
\end{lemma}

\begin{proposition}\label{propostion2.5}
    Suppose that $k\in\{1,2,3\}$ and $f_1,\ldots,f_n\in C^k(I)$ for equation \eqref{equation1}. Then the equation has at most $k$ periodic solution(s) in the region $[0,2\pi]\times I$,
    if $f_1^{(k)},\ldots,f_n^{(k)}$ are simultaneously non-negative (resp. non-positive) with at least one being positive (resp. negative) on $I$.
\end{proposition}
The proofs of Lemma \ref{han} and Proposition \ref{propostion2.5} follow directly from the approach in \cite{Lloyd} and simple calculations. For the sake of brevity and compactness of the main text, they are arranged in the Appendix \ref{A.1}.

\begin{proposition}\label{proposition2.6}
If $x=\lambda$ is a constant solution of equation \eqref{equation2}, then the Poincar\'{e} map $P(x_0)=x(1;0,x_{0})$ is $C^1$-differentiable at $\lambda$ and satisfies
\begin{align*}
  P'(\lambda)=\exp\left(\frac{1}{2}\left(f'_1(\lambda)+f'_2(\lambda)\right)\right).
\end{align*}
If additionally $f_1,f_2\in C^k(I)$, $P'(\lambda)=1$ and $P^{(2)}(\lambda)=\cdots=P^{(k-1)}(\lambda)=0$ for $k\in\{2,3\}$, then $P$ is $C^k$-differentiable at $\lambda$ and
\begin{align*}
  P^{(k)}(\lambda)=\left(f^{(k)}_1(\lambda)+f^{(k)}_2(\lambda)\right)\int_{0}^{\frac{1}{2}}\exp\left((k-1) f'_1(\lambda)t\right)dt.
\end{align*}
\end{proposition}

\begin{proof}
  The first assertion is trivial by applying statement (i) of Lemma \ref{han} to equation \eqref{equation2}. Next we verify the second assertion for $k=2$, and the same can be done for $k=3$ using a similar argument. Since $P'(\lambda)=1$ means $f'_2(\lambda)=-f'_1(\lambda)$, we get by assumption and statement (ii) of Lemma \ref{han} that $P$ is $C^2$-differentiable at $\lambda$ and satisfies
  \begin{align*}
    P''(\lambda)&=\int^{\frac{1}{2}}_{0}f''_1(\lambda)\exp\left(f'_1(\lambda)t\right)dt
                 +\int^{1}_{\frac{1}{2}}f''_2(\lambda)\exp\left(\frac{1}{2}f'_1(\lambda)+f'_2(\lambda)\Big(t-\frac{1}{2}\Big)\right)dt\\
                &=f''_1(\lambda)\int^{\frac{1}{2}}_{0}\exp\left(f'_1(\lambda)t\right)dt
                 +f''_2(\lambda)\int^{\frac{1}{2}}_{0}\exp\left(f'_1(\lambda)t\right)dt.
  \end{align*}
  Hence, the second assertion is true for $k=2$.
\end{proof}

\section{Proofs of Theorems \ref{th1'}, \ref{coro1} and \ref{th2}}
This section is divided into two parts. In the first subsection, we aim to give a result for the derivatives of the Poincar\'{e} map of equation \eqref{equation1} that includes Theorems \ref{th1'} and \ref{coro1} as the important cases. Then, we will additionally characterize the distribution of the periodic solutions of equation \eqref{equation2} in the second subsection, proving Theorem \ref{th2}.
\subsection{Derivative formulas in discrete form of Poincar\'{e} map}
We shall begin by recalling some notations defined in Theorem \ref{th1'} for the solution $x(t;t_0,x_{0})$ of equation \eqref{equation1}, that is,
\begin{align*}
V:=\left\{\rho\in \mathcal D\left|f_i\Big(x\Big(\frac{i-1}{n};0,\rho\Big)\Big)\neq0,\ i=1,2,\cdots,n\right.\right\}, \text{ and }
x_i:=x\Big(\frac{i}{n};0,x_0\Big),
\end{align*}
where $\mathcal D$ is the domain of the Poincar\'{e} map $P(x_0)=x(1;0,x_{0})$.
Also, we adopt the notation $P_i$, as in the proof of Lemma \ref{lemma1}, to represent the Poincar\'{e} map of equation \eqref{equation1} from $t=\frac{i-1}{n}$ to $t=\frac{i}{n}$, $i=1,2,\ldots,n$. Then, for $x_0\in V$,
\begin{align}\label{eq5}
P_i(x_{i-1})=x\Big(\frac{i}{n};\frac{i-1}{n},x_{i-1}\Big)=x_i,\indent i=1,2,\ldots,n.
\end{align}
We naturally have the decomposition $P=P_n\circ\cdots\circ P_1$.
\begin{lemma}\label{th1}
Suppose that $x_0\in V$. Then the Poincar\'{e} maps $P$ and $P_i$ (for $i=1,2,\ldots,n$) of equation \eqref{equation1} are $C^2$-differentiable at $x_0$ and $x_{i-1}$, respectively, satisfying
    \begin{align}
        &P^{'}(x_{0})=\prod_{i=1}^{n}P_{i}^{'}(x_{i-1}),
        \label{9}\\
        &P^{''}(x_{0})=P^{'}(x_{0})\cdot \sum_{i=1}^{n}\left(\frac{P^{''}_{i}(x_{i-1})}{P^{'}_{i}(x_{i-1})}\cdot \prod_{j=1}^{i-1}P^{'}_{j}(x_{j-1})\right),
        \label{8}
    \end{align}
and
\begin{equation}\label{eq6}
    P_{i}^{'}(x_{i-1})=\frac{f_{i}(x_{i})}{f_{i}(x_{i-1})},\quad P_{i}^{''}(x_{i-1})=\frac{f_{i}(x_{i})\big(f_{i}^{'}(x_{i})-f_{i}^{'}(x_{i-1})\big)}{f_{i}^{2}(x_{i-1})}.
\end{equation}
If additionally $f_{i}\in  C^{2}(I)$ for $i=1,2,\cdots, n$, then $P$ and $P_i$ are $C^3$-differentiable at $x_0$ and $x_{i-1}$, respectively, satisfying
\begin{align}\label{7}
    \begin{split}
    P^{'''}(x_{0})
    =&P'(x_0)\left[\frac{3}{2}\Bigg(\frac{P^{''}(x_{0})}{P^{'}(x_{0})}\Bigg)^2\right.\\
    &+\left. \sum_{i=1}^{n}\Bigg(\frac{P^{'''}_{i}(x_{i-1})}{P^{'}_{i}(x_{i-1})}-\frac{3}{2} \bigg(\frac{P^{''}_{i}(x_{i-1})}{P^{'}_{i}(x_{i-1})}\bigg)^{2}\Bigg) \Bigg(\prod_{j=1}^{i-1}P^{'}_{j}(x_{j-1})\Bigg)^{2}\right],
        \end{split}
    \end{align}
and
\begin{align}\label{eq8}
\begin{split}
    \frac{P^{'''}_{i}(x_{i-1})}{P^{'}_{i}(x_{i-1})}&-\frac{3}{2} \bigg(\frac{P^{''}_{i}(x_{i-1})}{P^{'}_{i}(x_{i-1})}\bigg)^{2}\\
    &=\frac{2f''_{i}(x_{i})f_{i}(x_{i})-\big(f'_{i}(x_{i})\big)^{2}-2f''_{i}(x_{i-1})f_{i}(x_{i-1})+\big(f'_{i}(x_{i-1})\big)^{2}}
     {2f_{i}^{2}(x_{i-1})}.
\end{split}
\end{align}
\end{lemma}

\begin{proof}
In order to prove the assertion, we first show the differentiability of $P_i$ on the set $V_i:=\{x_{i-1}|x_0\in V\}$ and verify the expressions in \eqref{eq6} and \eqref{eq8}, $i=1,2,\ldots,n$.

It is clear that $f_{i}(x(t;0,x_0))\neq 0$ for $(t,x_0)\in\big[\frac{i-1}{n},\frac{i}{n}\big]\times V$, $i=1,2,\ldots,n$.
Then, under the assumption of $x_0\in V$,
\begin{align}\label{eq7}
  \int_{x_{i-1}}^{x_i}\frac{dx}{f_i(x)}
  =\int_{\frac{i-1}{n}}^{\frac{i}{n}}\frac{dx(t;0,x_0)}{f_i(x(t;0,x_0))}
  =\int_{\frac{i-1}{n}}^{\frac{i}{n}}dt
  =\frac{1}{n},\indent i=1,2,\ldots,n.
\end{align}
In view of $x_i=P_i(x_{i-1})$ from \eqref{eq5}, this characterizes the map $P_i$ over $V_i$.
Note that $V_i$ is open. By the initial assumption $f_i\in C^1(I)$, we can take the derivative of \eqref{eq7} with respect to $x_{i-1}$, and then obtain
\begin{equation}\label{eq9}
    \frac{dx_{i}}{dx_{i-1}}\cdot \frac{1}{f_{i}(x_{i})}-\frac{1}{f_{i}(x_{i-1})}=0,\ \text{ i.e., }\
        P_{i}^{'}(x_{i-1})=\frac{dx_{i}}{dx_{i-1}}=\frac{f_{i}(x_{i})}{f_{i}(x_{i-1})}.
\end{equation}
Accordingly, $P_{i}$ is $C^2$-differentiable on $V_{i}$, satisfying
\begin{equation*}
    P_{i}^{''}(x_{i-1})
    =\frac{d}{dx_{i-1}}\left(\frac{f_{i}(x_{i})}{f_{i}(x_{i-1})}\right)
    =\frac{f_{i}(x_{i})\big(f_{i}^{'}(x_{i})-f_{i}^{'}(x_{i-1})\big)}{f_{i}^{2}(x_{i-1})}.
\end{equation*}
The expressions in \eqref{eq6} are valid.
If additionally $f_{i}\in  C^{2}(I)$, then it follows from \eqref{eq6} that $P_{i}$ is $C^3$-differentiable on $V_{i}$. Also the expression in \eqref{eq8} can be verified by a direct calculation using \eqref{eq6}.

Now let us consider the map $P$. It is easy to see that the assertion for the differentiability of $P$ on $V$ follows directly from the above argument and the decomposition $P=P_n\circ\cdots\circ P_1$, and so does that for equality \eqref{9} taking the chain rule for derivatives into account.
Also, by differentiating \eqref{9} we get
\begin{equation*}
    P^{''}(x_{0})=\prod_{i=1}^{n}P^{'}_{i}(x_{i-1})\cdot \sum_{i=1}^{n}\frac{P^{''}_{i}(x_{i-1})\cdot \frac{dx_{i-1}}{dx_{0}}}{P^{'}_{i}(x_{i-1})}
    =P^{'}(x_{0})\cdot \sum_{i=1}^{n}\left(\frac{P^{''}_{i}(x_{i-1})}{P^{'}_{i}(x_{i-1})}\cdot \prod_{j=1}^{i-1}P^{'}_{j}(x_{j-1})\right),
\end{equation*}
where in the second equality we use \eqref{9} again and the fact from \eqref{eq9} that
\begin{align}\label{10}
\frac{dx_{i-1}}{dx_0}=\prod_{j=1}^{i-1}\frac{dx_{j}}{dx_{j-1}}=\prod_{j=1}^{i-1}P^{'}_{j}(x_{j-1}),\indent i=1,2,\ldots,n.
\end{align}
Thus, equality \eqref{8} holds.

It remains to prove \eqref{7}. Similar to the process for \eqref{9}, we differentiate \eqref{8}, take into account \eqref{10}, and then obtain that
$P'''(x_0)=P''(x_0)M_1(x_0)+P'(x_0)\big(M_2(x_0)+M_3(x_0)\big)$, where
\begin{align}\label{11}
    \begin{split}
        &M_1(x_{0})
        =
        \sum_{i=1}^{n}\left(\frac{P^{''}_{i}(x_{i-1})}{P^{'}_{i}(x_{i-1})}\cdot \prod_{j=1}^{i-1}P^{'}_{j}(x_{j-1})\right),\\
        &M_2(x_{0})
        =
        \sum_{i=1}^{n}\left(\Bigg(\frac{P^{'''}_{i}(x_{i-1})}{P^{'}_{i}(x_{i-1})}- \bigg(\frac{P^{''}_{i}(x_{i-1})}{P^{'}_{i}(x_{i-1})}\bigg)^{2}\Bigg)
        \Bigg(\prod_{j=1}^{i-1}P^{'}_{j}(x_{j-1})\Bigg)^{2}\right),\\
        &M_3(x_{0})
        =
        \sum_{i=1}^{n}\left(\frac{P^{''}_{i}(x_{i-1})}{P^{'}_{i}(x_{i-1})}\cdot \prod_{j=1}^{i-1}P^{'}_{j}(x_{j-1})\cdot \sum_{l=1}^{i-1}\Bigg(\frac{P^{''}_{l}(x_{l-1})}{P^{'}_{l}(x_{l-1})}\cdot \prod_{k=1}^{l-1}P^{'}_{k}(x_{k-1})\Bigg)\right).
    \end{split}
    \end{align}
Note that, due to \eqref{8}, the functions $M_1$ and $M_3$ can be rewritten as $M_1(x_0)=\frac{P''(x_0)}{P'(x_0)}$ and
\begin{align*}
\begin{split}
&M_3(x_0)\\
&\ =
\sum_{1\leq l<i\leq n}
\Bigg(\frac{P^{''}_{i}(x_{i-1})}{P^{'}_{i}(x_{i-1})}\cdot \prod_{j=1}^{i-1}P^{'}_{j}(x_{j-1})\Bigg)
\Bigg(\frac{P^{''}_{l}(x_{l-1})}{P^{'}_{l}(x_{l-1})}\cdot \prod_{k=1}^{l-1}P^{'}_{k}(x_{k-1})\Bigg)\\
&\ =\frac{1}{2}
\left(\sum_{i=1}^{n}\Bigg(\frac{P^{''}_{i}(x_{i-1})}{P^{'}_{i}(x_{i-1})}\cdot \prod_{j=1}^{i-1}P^{'}_{j}(x_{j-1})\Bigg)\right)^2
-\frac{1}{2}
\sum_{i=1}^{n}\Bigg(\frac{P^{''}_{i}(x_{i-1})}{P^{'}_{i}(x_{i-1})}\cdot\prod_{j=1}^{i-1}P^{'}_{j}(x_{j-1})\Bigg)^2\\
&\ =\frac{1}{2}
\Bigg(\frac{P''(x_0)}{P'(x_0)}\Bigg)^2
-\frac{1}{2}
\sum_{i=1}^{n}\Bigg(\frac{P^{''}_{i}(x_{i-1})}{P^{'}_{i}(x_{i-1})}\cdot\prod_{j=1}^{i-1}P^{'}_{j}(x_{j-1})\Bigg)^2,
\end{split}
\end{align*}
respectively.
Hence, the substitutions of these identities in \eqref{11} immediately yield \eqref{7}.
This finishes the proof of the result.
\end{proof}

By means of Lemma \ref{th1}, we are able to easily prove Theorem \ref{th1'} and Theorem \ref{coro1}. We remark that the lemma actually provides an additional result (the second assertion), which we believe, could be useful in the development of further research on the issue.

\begin{proof}[Proof of Theorem \ref{th1'}]
The assertion readily follows from Lemma \ref{th1} by substituting \eqref{eq6} into \eqref{9} and \eqref{8}.
\end{proof}

\begin{proof}[Proof of Theorem \ref{coro1}]
It is a simple application of Theorem \ref{th1'} to equation \eqref{equation2}. In fact, if $x(t;0,x_0)$ is a non-constant periodic solution of the equation, then $x_2=x(1;0,x_0)=x_0$ and $f_i(x_j)\neq0$ for $i,j\in\{1,2\}$. Thus, the first assertion is directly obtained using Theorem \ref{th1'}. When $P'(x_0)=1$, we have by \eqref{12} in the theorem that
    \begin{align}\label{13}
        \begin{split}
            P^{''}(x_{0})=&\frac{f_{1}^{'}(x_{1})-f_{1}^{'}(x_{0})}{f_{1}(x_{0})}+\frac{f_{2}^{'}(x_{0})-f_{2}^{'}(x_{1})}{f_{2}(x_{1})}\cdot \frac{f_{1}(x_{1})}{f_{1}(x_{0})}\\
            =&
            \frac{1}{f_{1}(x_{0})f_{2}(x_{1})}
    \begin{vmatrix}
            f_{1}^{'}(x_{1})-f_{1}^{'}(x_{0}) & f_{2}^{'}(x_{1})-f_{2}^{'}(x_{0})\\
            f_{1}(x_{1}) & f_{2}(x_{1})
    \end{vmatrix}.
        \end{split}
        \end{align}
Moreover, from the first expression in \eqref{12}, $P'(x_0)=1$ also implies that $\frac{f_1(x_1)}{f_1(x_0)}=\frac{f_2(x_1)}{f_2(x_0)}$. This together with the first equality in \eqref{13} yields
\begin{equation*}
        \begin{split}
            P^{''}(x_{0})=&\frac{f_{1}^{'}(x_{1})-f_{1}^{'}(x_{0})}{f_{1}(x_{0})}+\frac{f_{2}^{'}(x_{0})-f_{2}^{'}(x_{1})}{f_{2}(x_{1})}\cdot \frac{f_{2}(x_{1})}{f_{2}(x_{0})}\\
            =
            & \frac{1}{f_{1}(x_{0})f_{2}(x_{0})}
    \begin{vmatrix}
            f_{1}^{'}(x_{1})-f_{1}^{'}(x_{0}) & f_{2}^{'}(x_{1})-f_{2}^{'}(x_{0})\\
            f_{1}(x_{0}) & f_{2}(x_{0})
    \end{vmatrix}.
        \end{split}
        \end{equation*}
Accordingly, the validity for the expressions of $P''(x_0)$ is proved and so is our result.
\end{proof}

\subsection{Characterization for the distribution of periodic solutions}
We now turn to the proof of Theorem \ref{th2}. In the following we will continue to use the notations $V$ and $x_i$'s defined in Theorem \ref{th1'}, specific to equation \eqref{equation2}. Also for $x_0\in V$, we denote by
\begin{align*}
  A_{x_0}=\big\{x(t;0,x_0)\big|t\in[0,1]\big\}\ \text{and}\
  A_{x_0,i}=\Big\{x(t;0,x_0)\Big|t\in\Big[\frac{i-1}{2},\frac{i}{2}\Big]\Big\},\indent i=1,2.
\end{align*}
\begin{proof}[Proof of Theorem \ref{th2}.]
        (i.1) The assertion in the statement is trivial.

        (i.2)
        Clearly, if $x(t;0,x_0)$ is a non-constant periodic solution of equation \eqref{equation2}, then $f_1|_{A_{x_0,1}}$ and $f_2|_{A_{x_0,2}}$ have definite and opposite signs. For the sake of clarity we only consider the case $f_1|_{A_{x_0,1}}>0$, and the other one follows in exactly the same argument.
        In this case, $A_{x_0}=A_{x_0,1}=A_{x_0,2}=[x_0,x_1]$ and therefore $f_1f_2<0$ on $A_{x_0}$. We have that $A_{x_0}\subset I\backslash\{x|f_1(x)f_2(x)=0\}=\bigcup^{m}_{i=0}I_i$. The first assertion holds.

        Next we prove the second assertion by contradiction.
        Let
        \begin{align*}
        H(x)=\frac{1}{f_1(x)}+\frac{1}{f_2(x)},\indent x\in \bigcup_{i=0}^{m}I_i.
        \end{align*}
        Suppose that $f_{1}+f_{2}$ does not change sign on $A_{x_0}$. Then,
        \begin{align}\label{14}
          \int_{x_0}^{x_1}H(x)dx
          =\int_{A_{x_0}}\frac{f_1(x)+f_2(x)}{f_1(x)f_2(x)}dx
          \neq0,
        \end{align}
        where in the second inequality we take into account the fact that $f_1+f_2\not\equiv0$ on $A_{x_0}$ from assumption.
        However, since $x(t;0,x_0)$ is periodic (i.e., $x_2=x(1;0,x_0)=x_0$),
        \begin{align*}
          \int_{x_0}^{x_1}H(x)dx
          =\int_{x_0}^{x_1}\frac{dx}{f_1(x)}-\int_{x_1}^{x_2}\frac{dx}{f_2(x)}
          =\int_{0}^{\frac{1}{2}}dt-\int_{\frac{1}{2}}^{1}dt
          =0.
        \end{align*}
        This contradicts \eqref{14}. Consequently, $f_1+f_2$ must be sign-changing on $A_{x_0}$.

        (ii) Set $F(x_0):=\int^{x_1}_{x_0}H(x)dx$.
        We continue to consider the function $F$ on $U:=\{\rho\in V|A_{\rho}\subset E\}$ and then verify the first assertion.
        It is easy to see that for $x_0\in U$,
        \begin{align}\label{18}
          F(x_0)=\int_{x_0}^{x_1}\frac{dx}{f_1(x)}-\int_{x_1}^{x_2}\frac{dx}{f_2(x)}+\int_{x_0}^{x_2}\frac{dx}{f_2(x)}=\int_{x_0}^{x_2}\frac{dx}{f_2(x)},
        \end{align}
        where the second equality follows from the identities $\int_{x_{0}}^{x_1}\frac{dx}{f_1(x)}=\int_{x_{1}}^{x_2}\frac{dx}{f_2(x)}=\frac{1}{2}$. Furthermore, according to Theorem \ref{th1'},
        \begin{align}\label{17}
        \begin{split}
          F'(x_0)
          &=\frac{dx_2}{dx_0}\cdot\frac{1}{f_2(x_2)}-\frac{1}{f_2(x_0)}
           =\frac{f_2(x_2)f_1(x_1)}{f_2(x_1)f_1(x_0)}\cdot\frac{1}{f_2(x_2)}-\frac{1}{f_2(x_0)}\\
          &
           =\frac{f_1(x_1)\big(f_1(x_0)+f_2(x_0)\big)-f_1(x_0)\big(f_1(x_1)+f_2(x_1)\big)}{f_1(x_0)f_2(x_0)f_2(x_1)}.
        \end{split}
        \end{align}

        Now assume for a contradiction that there exist two periodic solutions of equation \eqref{equation2} located in the region $\{(t,x)|t\in [0,1], x\in E\}$, with the initial values $p_{0}$ and $q_{0}$ such that $p_{0}<q_{0}$. Then, for each $x_0\in(p_0,q_0)$, the solution $x(t;0,x_0)$ of the equation is well-defined on $[0,1]$ and satisfies $x(t;0,p_0)< x(t;0,x_0)< x(t;0,q_0)$. We get that $[p_{0},q_{0}\big]\subset U$.
        The following argument will be still restricted to the case $f_1|_E>0$ because the case $f_1|_E<0$ can be treated similarly.
        Denote by $p_1=x(\frac{1}{2};0,p_0)$ and $q_1=x(\frac{1}{2};0,q_0)$. As in the proof of statement (i.2), one has $A_{p_0}=[p_0,p_1]$ and $A_{q_0}=[q_0,q_1]$.
        Let $a$ be the unique point in $\text{Int}(E)$ at which $f_1+f_2$ changes sign. Applying statement (i.2) shows that $a\in\text{Int}(A_{x_0})= (x_0,x_1)$ for $x_0\in\{p_0,q_0\}$. Thus, for $x_0\in (p_{0},q_{0}\big)$ we obtain
        \begin{align*}
          p_{0}<x_0<q_{0}<a<p_{1}<x_1<q_{1},
        \end{align*}
        which implies that $\big(f_1(x_0)+f_2(x_0)\big)\big(f_1(x_1)+f_2(x_1)\big)\leq0$. Note that $f_1+f_2\not\equiv0$ on $(p_{0},q_{0}\big)$ from assumption, and $f_1(x_0)f_1(x_1)>0$ due to $f_1|_E>0$.
        Hence, we know by \eqref{17} that $F'$ does not change sign and does not vanish identically on $[p_{0},q_{0}]$. This yields $F(p_0)\neq F(q_0)$, but contradicts $F(p_0)= F(q_0)=0$ from \eqref{18}. Consequently, equation \eqref{equation2} has at most one periodic solution in the region.

        Finally, if $x=x(t;0,x_0)$ is a periodic solution in $\{(t,x)|t\in [0,1], x\in E\}$, then it is non-constant. According to Theorem \ref{coro1},
        \begin{align*}
          P'(x_0)=1+\frac{f_1(x_1)\big(f_1(x_0)+f_2(x_0)\big)-f_1(x_0)\big(f_1(x_1)+f_2(x_1)\big)}{f_1(x_0)f_2(x_1)},
        \end{align*}
        where $P(x_0)=x(1;0,x_0)$. Recall that $\big(f_1(x_0)+f_2(x_0)\big)\big(f_1(x_1)+f_2(x_1)\big)\leq0$ and $f_1(x_0)f_1(x_1)>0$ from the above argument. Thus, when $f_1+f_2$ has only one zero on $E$, we have that $P'(x_0)\neq1$, i.e., the periodic solution is hyperbolic.
\end{proof}

\section{Applications to single-species models}
We are going to apply our main results to the study of limit cycles of models \eqref{mosquito}, \eqref{xiaodongmei} and \eqref{main equation}, respectively, following the procedure introduced in Section 1.

\subsection{Application 1}
We first prove Theorem \ref{application1}.
The argument can be restricted to equation \eqref{xiaodongmei} in one period (with $k=0$). Using Lemma \ref{lemma1}, the equation is normalized to
\begin{equation}\label{xiaodongmei2}
    \frac{dx}{dt}=
    \left\{
  \begin{aligned}
        &g_1(x)=2T_{1}g(x),&t\in \Big[0,\frac{1}{2}\Big), x\in\mathbb R_0^+,\\
        &g_2(x)=2(T-T_{1})(g(x)-h),&t\in \Big[\frac{1}{2},1\Big], x\in\mathbb R_0^+,
    \end{aligned}
  \right.
\end{equation}
where $h>0$, $T>T_1>0$ and $g$ satisfies hypothesis (H).
In the following we pre-analyze equation \eqref{xiaodongmei2} in two steps.\vskip0.2cm

\textbf{Step 1:} Determine the common zeros of $g_{1}$ and $g_{2}$, and obtain $\mathcal U=\mathbb R_0^+\backslash\{x|g_1(x)g_2(x)=0\}$.

According to hypothesis (H), the function $g$ is strictly increasing on $[0,k_0]$ and strictly decreasing on $[k_0,+\infty)$, respectively. Let $h^*$ be defined as in Theorem \ref{application1}, i.e., $h^*=g(k_0)$. Then, the following facts are clear (see Figure \ref{fig-app1} for illustration):
\begin{itemize}
    \item $g_{1}$ has exactly two zeros, $x=0$ and $x=K$.
    \item If $0<h<h^*$, then $g_{2}$ has exactly two zeros, denoted by $x=\lambda_{1}(h)$ and $x=\lambda_{2}(h)$, satisfying $\lambda_{1}\in(0,k_{0})$ and $\lambda_{2}\in(k_0,K)$.
    \item If $h=h^*$, then $g_{2}$ has only one zero, $x=k_{0}$.
    \item If $h>h^*$, then $g_{2}$ has no zeros.
\end{itemize}
Thus, there are no common zeros of $g_{1}$ and $g_{2}$, i.e., 
equation \eqref{xiaodongmei2} has no constant limit cycles. And we have that
\begin{itemize}
  \item[(a)] If $0<h<h^*$, then $\mathcal U=(0,\lambda_1)\cup(\lambda_1,\lambda_2)\cup(\lambda_2,K)\cup(K,+\infty)$.
  \item[(b)] If $h=h^*$, then $\mathcal U=(0,k_0)\cup(k_0,K)\cup(K,+\infty)$.
  \item[(c)] If $h>h^*$, then $\mathcal U=(0,K)\cup(K,+\infty)$.
\end{itemize}

\textbf{Step 2:} Determine the sign changes of $g_{1}+g_{2}$.

Since $g_{1}(x)+g_{2}(x)=2Tg(x)-2(T-T_{1})h$, the following facts are also clear from the monotonicity of $g$ and a simple calculation  (see again Figure \ref{fig-app1} for illustration):
\begin{itemize}
    \item[(d)] If $0<h<\frac{T}{T-T_{1}}h^*$, then $g_1+g_2$ has exactly two zeros, denoted by $x=\mu_1(h)$ and $x=\mu_2(h)$, satisfying
    \begin{itemize}
      \item[(d.1)] $\mu_{1}\in(0,\lambda_{1})$ and $\mu_2\in(\lambda_{2},K)$, if additionally $0<h<h^*$,
      \item[(d.2)] $\mu_{1}\in(0,k_{0})$ and $\mu_2\in(k_{0},K)$, if additionally $h^*\leq h<\frac{T}{T-T_{1}}h^*$.
    \end{itemize}
    Furthermore, $g_{1}+g_{2}$ is positive on $(\mu_{1},\mu_{2})$ and negative on $(0,\mu_{1})\cup (\mu_{2},+\infty)$, respectively.
    \item[(e)] If $h\geq\frac{T}{T-T_{1}}h^*$, then $g_{1}+g_{2}\leq 0$.
\end{itemize}
Hence, on account of statements (a)--(e), the number of sign changes of $g_1+g_2$ on the connected components of $\mathcal U$, denoted by $s_{(\cdot,\cdot)}$, can be summarized in Table \ref{table4} below.
\begin{table}[!ht]\vspace{-0.2cm}
        \centering
        \renewcommand\arraystretch{1.5}
        \begin{threeparttable}
        \begin{tabular}{|c|c|}
            \hline
             $h$ & $s_{(\cdot,\cdot)}$\\
            \hline
           $0<h<h^*$ & $s_{(0,\lambda_1)}=s_{(\lambda_2,K)}=1$ and $s_{(\lambda_1,\lambda_1)}=s_{(K,+\infty)}=0$ \\
            \hline
           $h=h^*$ & $s_{(0,k_0)}=s_{(k_0,K)}=1$ and $s_{(K,+\infty)}=0$\\
            \hline
           $h^*<h<\frac{T}{T-T_{1}}h^*$ & $s_{(0,K)}=2$ and $s_{(K,+\infty)}=0$ \\
           \hline
           $h\geq\frac{T}{T-T_{1}}h^*$ & $s_{(0,K)}=s_{(K,+\infty)}=0$ \\
           \hline
        \end{tabular}
        \end{threeparttable}
        \caption{The number of sign changes of $g_1+g_2$.}\label{table4}
        \vspace{-0.6cm}
    \end{table}

\begin{figure}[!htbp]
    \centering
    \begin{subfigure}[b]{0.48\textwidth}
        \includegraphics[scale=0.39]{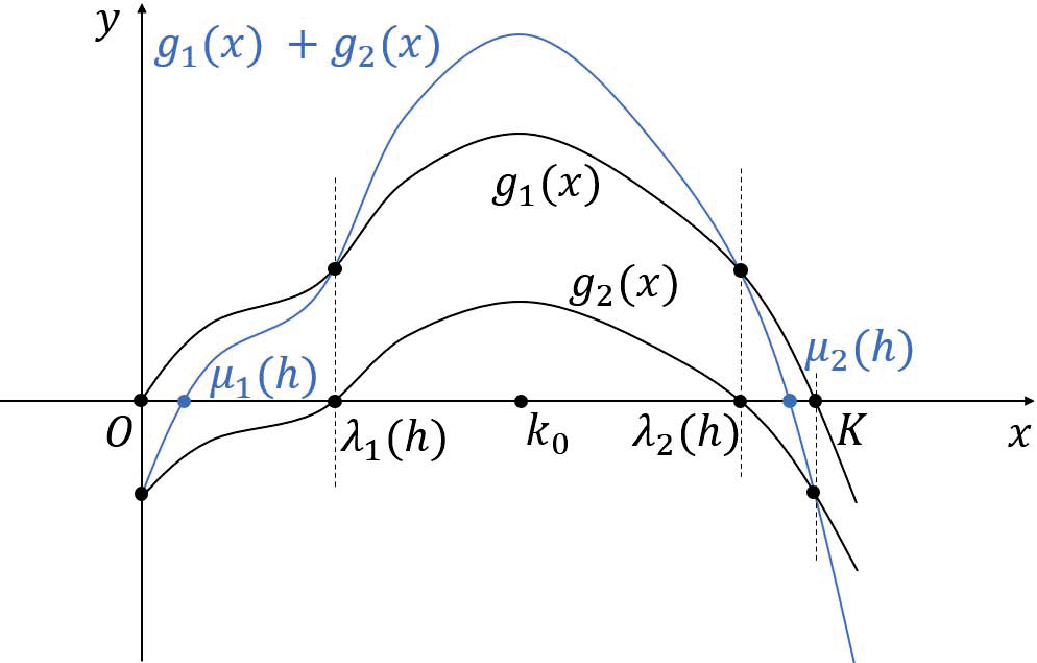}
        \caption{$0<h<h^*$}
    \end{subfigure}
    \begin{subfigure}[b]{0.48\textwidth}
            \includegraphics[scale=0.39]{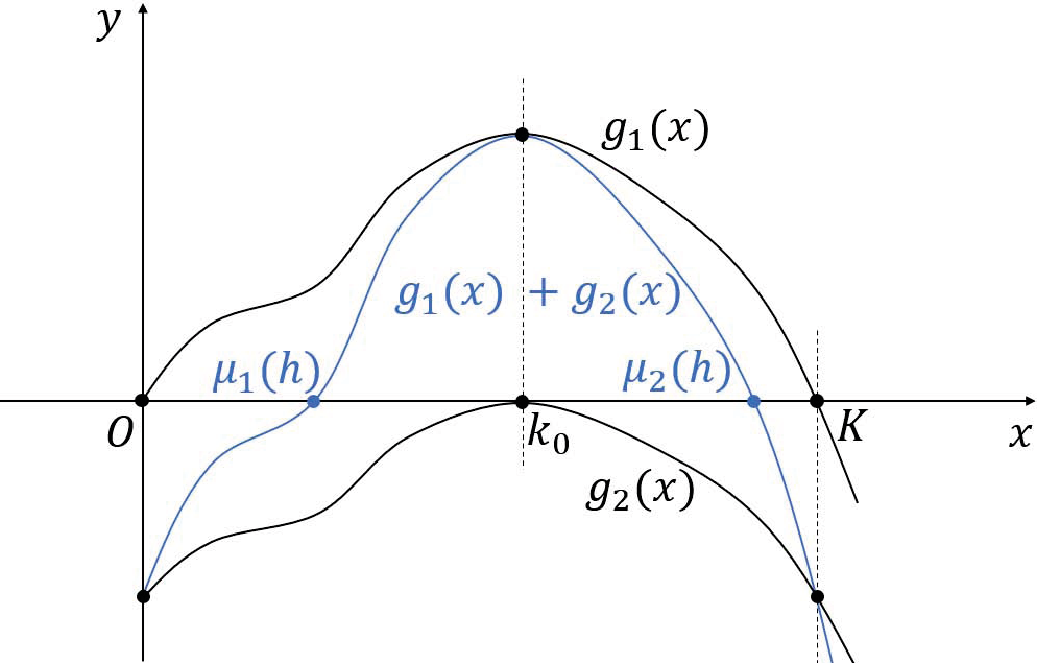}
            \caption{$h=h^*$}
    \end{subfigure}
    \begin{subfigure}[b]{0.48\textwidth}
        \includegraphics[scale=0.39]{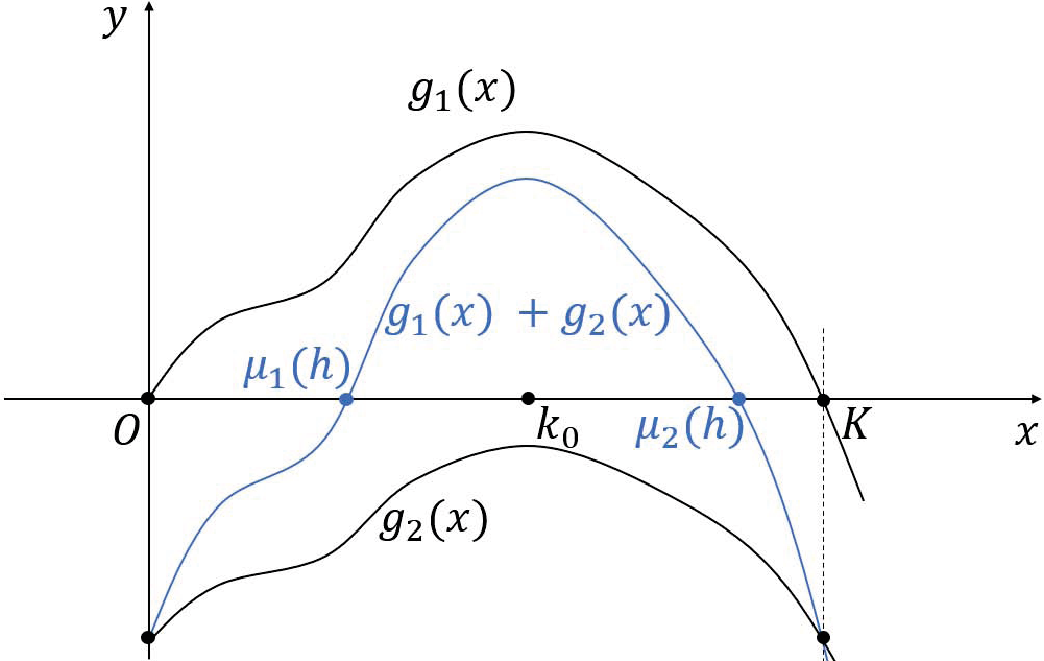}
        \caption{$h^*<h<\frac{T}{T-T_1}h^*$}
    \end{subfigure}
    \begin{subfigure}[b]{0.48\textwidth}
            \includegraphics[scale=0.39]{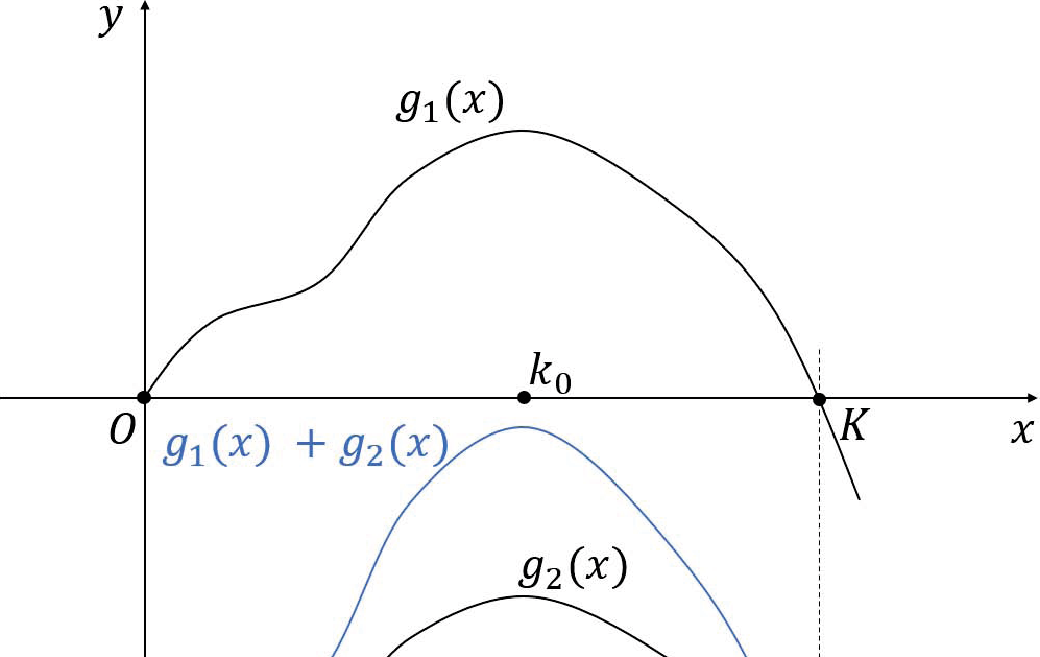}
            \caption{$\frac{T}{T-T_1}h^*\leq h$}
    \end{subfigure}
    \caption{The relative positions of the zeros of $g_1$, $g_2$, and $g_1+g_2$ for different values of $h$.}\label{fig-app1}
\end{figure}

We are now able to give the following result using Theorems \ref{coro1} and \ref{th2}.
\begin{proposition}\label{prop1-1}
    Under hypothesis $\mathrm{(H)}$, the following statements hold.
    \begin{itemize}
        \item[(i)]If $0< h\leq h^*$, then equation \eqref{xiaodongmei2} has exactly two limit cycles, where the lower one is unstable and the upper one is stable.
        \item[(ii)]If $h^*<h< \frac{T}{T-T_{1}}h^*$, then equation \eqref{xiaodongmei2} has at most two limit cycles, counted with multiplicities.
        \item[(iii)]If $h\geq \frac{T}{T-T_{1}}h^*$, then equation \eqref{xiaodongmei2} has no limit cycles.
    \end{itemize}
\end{proposition}

\begin{proof}
According to the argument in Step 1 and statement (i.2) of Theorem \ref{th2}, the limit cycles of equation \eqref{xiaodongmei2} can only appear in the region $[0,1]\times\mathcal U$. In what follows we shall verify the statements one by one.

    (i) It is a direct conclusion from Table \ref{table4} and Theorem \ref{th2} that equation \eqref{xiaodongmei2} has at most two limit cycles. Furthermore, one can easily check by hypothesis (H) that: $g_1(0)=0$ and $g_2(0)<0$; $g_1(k_0)>0$ and $g_2(k_0)\geq0$; $g_1(K)=0$ and $ g_2(K)<0$. These yield an unstable limit cycle in $[0,1]\times(0,k_0)$ and a stable limit cycle in $[0,1]\times(k_0,K)$ of the equation. Thus, statement (i) holds.

    (ii)
    In this case, Table \ref{table4} and statement (i.2) of Theorem \ref{th2} indicate that the limit cycles of equation \eqref{xiaodongmei2} are further restricted to the region $[0,1]\times(0,K)$.
    To prove the assertion of the statement we will utilize Theorem \ref{coro1}. Let $x(t;0,x_0)$ be the solution of the equation with initial value $x(0;0,x_0)=x_0$. For convenience we use the notations $x_1=x(\frac{1}{2};0,x_0)$ and $P(x_0)=x(1;0,x_0)$ as in Theorem \ref{coro1}. According to the theorem and a simple calculation, if $x=x(t;0,x_0)$ is periodic and located in $[0,1]\times(0,K)$, then
    \begin{equation}\label{signP^{'}}
        P'(x_{0})=1+\frac{4(T-T_{1})T_{1}h}{g_{1}(x_{0})g_{2}(x_{1})}\Big(g(x_{0})-g(x_{1})\Big),
    \end{equation}
    and
    \begin{equation}\label{signP''}
        P''(x_{0})=\frac{4(T-T_{1})T_{1}h}{g_{1}(x_{0})g_{2}(x_{1})}\left(g^{'}(x_{0})-g^{'}(x_{1})\right) \text{ when } P'(x_{0})=1.
    \end{equation}
    Note that $g_1>0$ and $g_2<0$ on $(0,K)$ when $h>h^*$. Hence, the multiplicity and stability of the periodic solution, can be determined by $G_0(x_0):=g(x_0)-g(x_1)$ and $G_1(x_0):=g'(x_0)-g'(x_1)$, provided that they do not vanish simultaneously.

    Now we focus on the functions $G_0$ and $G_1$. Since $g_1(0)=g_1(K)=0$, it follows that $x_1:x_0\mapsto x(\frac{1}{2};0,x_0)$ forms an isomorphism on $[0,K]$.  Thus, $G_0$ and $G_1$ are well-defined on $[0,K]$.
    We claim that there exists $a\in(0,K)$, such that $G_0<0$ on $(0,a)$, $G_0=0$ and $G_1>0$ at $x_0=a$, and $G_0>0$ on $(a,K)$.
    Indeed, let $x_*\in(0,K)$ be the point satisfying $x(\frac{1}{2};0,x_*)=k_0$. We have $x_*<k_0$ and the following three observations, taking into account the fact that $g_1>0$ on $(0,K)$ and the monotonicity of $g$:
    \begin{itemize}
      \item  If $x_0\in(0,x_*]$, then $x_0<x_1\leq x(\frac{1}{2};0,x_*)=k_0$. Therefore $G_0<0$ on $(0,x_*]$.
      \item  If $x_0\in(x_*,k_0)$, then $x_0<k_0=x(\frac{1}{2};0,x_*)<x_1$. Therefore $G_0$ is strictly increasing and $G_1>0$ on $(x_*,k_0)$.
      \item  If $x_0\in[k_0,K)$, then $k_0\leq x_0<x_1$. Therefore $G_0>0$ on $[k_0,K)$.
    \end{itemize}
    These immediately lead to the existence of $a$. The claim is proved.

    Let us come back to the estimate for the number of limit cycles of equation \eqref{xiaodongmei2}. From \eqref{signP^{'}} and the above claim, equation \eqref{xiaodongmei2} has at most one non-hyperbolic limit cycle, which is determined by the initial value $x_0=a$ if it exists. In such a case, we know by \eqref{signP''} and the claim that $P''(a)<0$. Hence, this non-hyperbolic limit cycle is upper-stable and lower-unstable, with the multiplicity being two. On the other hand, \eqref{signP^{'}} together with the claim also shows that any limit cycle of the equation with initial value $x_0<a$ (resp. $>a$) must be hyperbolic and unstable (resp. stable). Taking into account the fact that two consecutive limit cycles must possess opposite stability on the adjacent side, equation \eqref{xiaodongmei2} has at most two limit cycles, counted with multiplicities.

    (iii) It is also a direct conclusion applying Table \ref{table4} and statement (i.2) of Theorem \ref{th2}.
\end{proof}

As shown in Proposition \ref{prop1-1}, we have actually proved Theorem \ref{application1} except for the explicit evolution of the limit cycles when $h\in\big(h^*,\frac{T}{T-T_1}h^*\big)$.
We are going to give a small improvement for the proposition by utilizing the theory of rotated differential equations introduced in Section 2, and then finally complete the proof of the theorem.

\begin{proof}[Proof of Theorem \ref{application1}:]
    The result will follow from Proposition \ref{prop1-1} once the evolution of the limit cycles of equation \eqref{xiaodongmei2} is clarified for $h\in\big(h^*,\frac{T}{T-T_1}h^*\big)$. To finish this part, we first know by
    Definition \ref{definition of rotated} that equation \eqref{xiaodongmei2} forms a family of rotated equations on $[0,1]\times\mathbb R^+_0$ with respect to $h$.
    Then, from Lemma \ref{property of rotated}, the unstable limit cycle and the stable limit cycle of the equation given in statement (i) of Proposition \ref{prop1-1}, still exist for $h>h^*$ and are monotonically increasing and decreasing in $h$, respectively, unless there is a change in stability.
    Together with statements (ii) and (iii) of Proposition \ref{prop1-1}, these two limit cycles approach each other as $h$ increases, and then coincide to form a lower-unstable and upper-stable limit cycle at a specific value $h=h_{MSY}\in\big(h^*,\frac{T}{T-T_1}h^*\big)$. Note that the solution of this family of rotated equations, denoted by $x_h(t;0,x_0)$, is monotonically decreasing in $h$. Hence, equation \eqref{xiaodongmei2} has no limit cycles for $h>h_{MSY}$ because, once $x_h(1;0,x_0)$ is well-defined, $x_h(1;0,x_0)<x_{h_{MSY}}(1;0,x_0)\leq x_0$. Consequently, the assertion of the theorem holds for $h\in\big(h^*,\frac{T}{T-T_1}h^*\big)$. The proof is finished.
\end{proof}

\subsection{Application 2}
In this application we aim to prove Theorem \ref{main theorem}. It is sufficient to focus on equation \eqref{main equation} in one period (again taking $k=0$). Thus, due to Lemma \ref{lemma1} and the arbitrariness of the parameters $a_i$'s, $b_i$'s, $c_i$'s in the equation, the validity of the theorem can be known by considering the following equation
\begin{align}\label{main equation1}
\begin{split}
    \frac{dx}{dt}
    =
    \left\{
  \begin{aligned}
         &f_1(x)=a_{1}x^3+b_{1}x^2+c_{1}x,&  t\in \Big[0,\frac{1}{2}\Big),\ x\in\mathbb R,\\
         &f_2(x)=a_{2}x^3+b_{2}x^2+c_{2}x,&  t\in \Big[\frac{1}{2},1\Big],\ x\in\mathbb R,
    \end{aligned}
  \right.
\end{split}
\end{align}
where $a_{i}, b_{i}, c_{i}\in\mathbb R$ for $i=1,2$.

Similar to the argument in the previous Application 1, we begin by pre-analyzing equation \eqref{main equation1} in two steps. We will only consider the case $f_1+f_2\not\equiv0$ because, when $f_1+f_2\equiv0$, any solution $x(t;0,x_0)$ of equation \eqref{main equation1} well-defined on $[0,1]$ satisfies $x(t;0,x_0)=x(1-t;0,x_0)$, and this symmetry yields a periodic annulus but not limit cycles of the equation.
\vskip0.2cm

\textbf{Step 1:} Characterize the common zeros of $f_{1}$ and $f_{2}$ (i.e., the constant limit cycles), and the set $\mathcal U=\mathbb R\backslash\{x|f_1(x)f_2(x)=0\}$.

Clearly, a point $x=\lambda$ is a common zero of $f_1$ and $f_2$, if and only if it is a zero of $f_1+f_2$ on the set $\{x|f_1(x)f_2(x)=0\}$. Note that the multiplicity of any zero of $f_1+f_2$ does not exceed three.
Thus, we have the following simple observation taking Proposition \ref{proposition2.6} into account:
\begin{itemize}
 \item[(a)] $x=\lambda$ forms a constant limit cycle of equation \eqref{main equation1} with multiplicity $k$, if and only if it is a zero of $f_1+f_2$ with multiplicity $k$ on the set $\{x|f_1(x)f_2(x)=0\}$. Furthermore, $k\in\{1,2,3\}$.
\end{itemize}
Also it is easy to see that
\begin{itemize}
  \item[(b)] Equation \eqref{main equation1} always possesses a constant limit cycle $x=0$.
  \item[(c)] The set $\mathcal U$ consists of $2\leq m\leq6$ connected components $I_1,\ldots,I_m$. 
\end{itemize}

\textbf{Step 2:} Determine the sign-changes of $f_{1}+f_{2}$.

Let $N_c$ and $N_{\mathcal U}$ be the number of constant limit cycles of equation \eqref{main equation1} and the number of zeros of $f_1+f_2$ on $\mathcal U$, respectively, counted with multiplicities. According to statements (a), (b) and the degree of the polynomial $f_1+f_2$, we get that
$$1\leq N_c+N_{\mathcal U}\leq3.$$
This estimate together with statement (c) yields the following table for $s_E$, the number of sign-changes of $f_1+f_2$ on $E\in\{I_1,\ldots,I_m\}$.
\begin{table}[!ht]\vspace{-0.2cm}
        \centering
        \renewcommand\arraystretch{1.5}
        \begin{threeparttable}
        \begin{tabular}{|c|c|c|}
            \hline
             $N_c$ & $N_{\mathcal U}$  & $s_E$\\
            \hline
           $=3$ & $=0$ & $s_E=0$ for all $E\in \{I_1,\ldots,I_m\}$ \\
            \hline
           $=2$ & $\leq1$ & $s_E=1$ for at most one $E\in\{I_1,\ldots,I_m\}$, and $s_E=0$ for the others\\
            \hline
           \multicolumn{1}{|c|}{\multirow{2}{*}{\makecell[c]{=1}}} & \multicolumn{1}{c|}{\multirow{2}{*}{\makecell[c]{$\leq2$}}} & $s_E=1$ for at most two $E\in\{I_1,\ldots,I_m\}$, and $s_E=0$ for the others \\
           \cline{3-3}
            &  & $s_E=2$ for at most one $E\in\{I_1,\ldots,I_m\}$, and $s_E=0$ for the others \\
           \hline
        \end{tabular}
        \end{threeparttable}
        \caption{The number of sign changes of $f_1+f_2$.}\label{table5}
        \vspace{-0.6cm}
    \end{table}

From statement (i.2) of Theorem \ref{th2}, the non-constant limit cycles of equation \eqref{main equation1} are all located in the region $[0,1]\times\bigcup^{m}_{i=1}I_i$.
Now we estimate the number of non-constant limit cycles of the equation based on Table \ref{table5}, Theorem \ref{coro1} and Theorem \ref{th2}. Our results are summarized in two propositions. The first one is for the case $N_c\in\{2,3\}$.

\begin{proposition}\label{prop1}
    If $f_1+f_2\not\equiv0$ and $N_c\in\{2,3\}$, then equation \eqref{main equation1} has at most one non-constant limit cycle, counted with multiplicity.
\end{proposition}
\begin{proof}
     The conclusion is immediately obtained by Table \ref{table5} and Theorem \ref{th2}.
\end{proof}

To give the result for the second case $N_c=1$, we start with the estimate for the multiplicity of non-constant limit cycles of equation \eqref{main equation1}. Let us consider the Poincar\'{e} map $P(x_0)=x(1;0,x_0)$ of the equation given by the solution $x(t;0,x_{0})$ with initial value $x(0;0,x_0)=x_0$. When $x=x(t;0,x_0)$ is a non-constant limit cycle, it follows from Theorem \ref{coro1} and the linearity of determinant in rows that
\begin{equation}\label{P'}
        P'(x_{0})=1+\frac{x_{0}x_{1}(x_{1}-x_{0})}{f_{1}(x_{0})f_{2}(x_{1})}\Big(Ax_{0}x_{1}-B(x_{0}+x_{1})-C\Big),
    \end{equation}
    where $x_{1}=x(\frac{1}{2};0,x_{0})$ and
\begin{align}\label{eq10}
  A=
  \left|
  \begin{array}{cc}
  a_1   & a_2\\
  b_1   & b_2  \\
  \end{array}
  \right|,\ \
  B=
  \left|
  \begin{array}{cc}
  c_1   & c_2\\
  a_1   & a_2  \\
  \end{array}
  \right|,\ \
  C=
  \left|
  \begin{array}{cc}
  c_1   & c_2\\
  b_1   & b_2  \\
  \end{array}
  \right|.
\end{align}
Furthermore, observe that the non-constant assumption for the limit cycle ensures that $x_0x_1>0$ and $x_1-x_0\neq 0$. One can obtain
\begin{align}\label{eq13}
Ax_{0}x_{1}-B(x_{0}+x_{1})-C=0 \text{ if and only if } P'(x_0)=1.
\end{align}
This together with the second conclusion of Theorem \ref{coro1} additionally yields
\begin{equation}\label{P''}
        P''(x_{0})
        =\frac{x_{1}(x_{1}^{2}-x_{0}^{2})}{f_{1}(x_{0})f_{2}(x_{1})}(Ax_{1}-B)
        =\frac{x_{0}(x_{1}^{2}-x_{0}^{2})}{f_{1}(x_{0})f_{2}(x_{0})}(Ax_{0}-B)\
        \text{when}\ P'(x_0)=1.
\end{equation}
We have the following auxiliary lemma.

\begin{lemma}\label{prop5}
    Suppose  that $f_1+f_2\not\equiv0$. Then the multiplicity of any non-constant limit cycle of equation \eqref{main equation1} is at most two.
\end{lemma}

\begin{proof}
    This is a simple assertion using \eqref{eq10}, \eqref{eq13} and \eqref{P''}.
    Suppose that $x=x(t;0,x_{0})$ is a non-constant limit cycle of equation \eqref{main equation1} with multiplicity greater than two. Then according to \eqref{eq13}, \eqref{P''} and the above observation for $x_0$ and $x_1$, we have
    \begin{align*}
      Ax_{0}x_{1}-B(x_{0}+x_{1})-C=0\ \
      \text{and}\ \
      Ax_{0}-B=Ax_{1}-B=0.
    \end{align*}
    This implies $A=B=C=0$. Hence, the parameter vectors $(a_1,b_1,c_1)$ and $(a_2,b_2,c_2)$ are linearly dependent because, due to \eqref{eq10}, their cross product is zero. In other words, there exists $(\alpha_1,\alpha_2)\in\mathbb R^2\backslash\{(0,0)\}$ such that $\alpha_1 f_1+\alpha_2 f_2\equiv0$. However, in this case either $f_1+f_2=\frac{\alpha_2-\alpha_1}{\alpha_2}f_1$ or $f_1+f_2=\frac{\alpha_1-\alpha_2}{\alpha_1}f_2$, and $\alpha_1\neq\alpha_2$ under assumption. Thus, $f_1+f_2$ is nonzero on $\mathcal U=\bigcup^{m}_{i=1}I_i$, which contradicts statement (i.2) of Theorem \ref{th2} for $x(t;0,x_0)$. Accordingly, the multiplicity of $x=x(t;0,x_0)$ must be at most two.
\end{proof}

\begin{proposition}\label{prop2}
    If $f_1+f_2\not\equiv0$ and $N_c=1$, then equation \eqref{main equation1} has at most two non-constant limit cycles, counted with multiplicities.
\end{proposition}

\begin{proof}
From Table \ref{table5} and Theorem \ref{th2}, the result will follow once we show that, for the case $E\in\{I_1,\ldots,I_m\}$ with $s_E=2$, equation \eqref{main equation1} has at most two limit cycles (counted with multiplicities) in the region $[0,1]\times E$. In the following, we prove this in two subcases according to the value of $A$ defined in \eqref{eq10}.
\vskip0.2cm

\noindent{\em Subcase 1}: $A=0$.

In this subcase, the formulas \eqref{P'} and \eqref{P''} for the limit cycle $x=x(t;0,x_0)$ and the Poincar\'{e} map $P$, are reduced to
    \begin{equation}\label{eq11}
        P^{'}(x_{0})=1-\frac{x_{0}x_{1}(x_{1}-x_{0})}{f_{1}(x_{0})f_{2}(x_{1})}\Big(B(x_{0}+x_{1})+C\Big),
    \end{equation}
    and
    \begin{equation}\label{eq12}
        P^{''}(x_{0})
        =\frac{x_{0}(x_{1}^2-x_{0}^2)}{f_{1}(x_{0})f_{2}(x_{0})}B\ \
        \text{when}\ \ P^{'}(x_{0})=1,
    \end{equation}
    where $B,C$ are defined in \eqref{eq10} and $x_1=x(\frac{1}{2};0,x_0)$.

    It is clear that $f_1$ and $f_2$ have definite signs on $E$, because they are nonzero on $\bigcup^{m}_{i=1}I_i$. Thus $x_1-x_0$ does not change sign for all $x_0\in\hat E:=\big\{\rho\in E\big|x(\frac{1}{2};0,\rho)\in E\big\}$. Also we have $x_0x_1>0$ for $x_0\in \hat E$. Hence, from Lemma \ref{prop5} and \eqref{eq12}, all the non-hyperbolic limit cycles of the equation in $[0,1]\times E$ have multiplicity two with the same semi-stability.

    Now suppose that equation \eqref{main equation1} has $r$ limit cycles in $[0,1]\times E$, counted with multiplicities. According to Definition \ref{definition of rotated}, the equation forms a family of rotated equations on $[0,1]\times E$ with respect to the parameter $c_1$. By the above argument and statement (ii) of Lemma \ref{property of rotated}, all the limit cycles with multiplicity two in the region split simultaneously into two hyperbolic limit cycles in an appropriate monotonic variation of $c_1$. Moreover, the equation always satisfies $A=0$ as $c_1$ varies. For these reasons, it is sufficient to consider the case where these $r$ limit cycles are all hyperbolic. Note that $x_1$ is increasing in $x_0$. Then the function $B(x_0+x_1)+C$ is monotonic on $\hat E$ and therefore can change sign at most once. Due to \eqref{eq11} and the fact that two consecutive hyperbolic limit cycles must possess opposite stability, we obtain $r\leq2$. The assertion in this subcase is proved.

\vskip0.2cm

\noindent{\em Subcase 2}: $A\neq0$.

First, from the analysis in subcase 1, there exists a constant $\alpha\in\{-1,1\}$ such that
\begin{align*}
  \text{sgn}\left(\frac{x_{0}x_{1}(x_{1}-x_{0})}{f_{1}(x_{0})f_{2}(x_{1})}\right)
  =\text{sgn}\left(\frac{x_{1}(x^2_{1}-x^2_{0})}{f_{1}(x_{0})f_{2}(x_{1})}\right)
  =\text{sgn}\left(\frac{x_{0}(x^2_{1}-x^2_{0})}{f_{1}(x_{0})f_{2}(x_{0})}\right)
  =\alpha,\ x_0\in\hat E.
\end{align*}
Let $G$ be a function on $\hat E$ given by $G(x_0)=Ax_{0}x_{1}-B(x_{0}+x_{1})-C$.
Then, for the limit cycle with initial value $x_0\in\hat E$ (i.e., the limit cycle in $[0,1]\times E$), we get by \eqref{P'} and \eqref{P''} that
\begin{align}\label{eq14}
  \text{sgn}\left(P'(x_0)-1\right)=\alpha\cdot\text{sgn}\big(G(x_0)\big),
\end{align}
and
\begin{align}\label{eq15}
  \text{sgn}\left(P''(x_0)\right)=\alpha\cdot\text{sgn}(Ax_1-B)=\alpha\cdot\text{sgn}(Ax_0-B) \ \
        \text{when}\ \ P'(x_0)=1.
\end{align}
Accordingly, when this limit cycle has multiplicity two, it satisfies
  \begin{align}\label{eq17}
      G(x_{0})=0\ \ \text{and}\ \ (Ax_{0}-B)(Ax_{1}-B)>0.
    \end{align}

     Without loss of generality, suppose that $A>0$ (otherwise take the change of variable $x\mapsto -x$ in the equation).
     We will finish the rest of the proof by four claims, which are stated one by one below.
\vskip0.1cm

    {\em Cliam 1}: Equation \eqref{main equation1} satisfies $B^{2}+AC>0$ if it has limit cycle(s) with multiplicity two in $[0,1]\times E$. Furthermore, every limit cycle of this type is specifically located in either $[0,1]\times E^+$ or $[0,1]\times E^-$, where $E^+=E\cap\big(\frac{B}{A},+\infty\big)$ and $E^-=E\cap\big(-\infty,\frac{B}{A}\big)$.
\vskip0.1cm

    In fact, let us rewrite the function $G$ as
    \begin{align}\label{eq16}
      G(x_0)=\frac{1}{A}\left((Ax_0-B)(Ax_1-B)-B^2-AC\right).
    \end{align}
    When $x_0$ is the initial value of the limit cycle with multiplicity two in $[0,1]\times E$, we immediately get from \eqref{eq17} and \eqref{eq16} that
    $B^2+AC=(Ax_{0}-B)(Ax_{1}-B)>0$. Also, this inequality indicates that either $x_0,x_1\in E^+$, or $x_0,x_1\in E^-$. Note that one of $x_0$ and $x_1$ is the maximum and the other is the minimum of $x(t;0,x_0)$.  Therefore, the limit cycle must be located in either $[0,1]\times E^+$ or $[0,1]\times E^-$. Claim 1 follows.
\vskip0.1cm

    {\em Cliam 2}: Equation \eqref{main equation1} has at most one limit cycle with multiplicity two in $[0,1]\times E$.
\vskip0.1cm

    To verify the claim, we introduce a function by \eqref{eq16}:
    \begin{align}\label{eq20}
      G_1(x_0)=\frac{AG(x_0)}{Ax_0-B}=Ax_1-B-\frac{B^2+AC}{Ax_0-B},\ \ \ x_0\in \hat E.
    \end{align}
    Suppose that there exist two limit cycles with multiplicity two of the equation in the region.
    By \eqref{eq17}, both initial values of these limit cycles are zeros of $G_1$ in $\hat E\cap(E^+\cup E^-)$.
    Recall that $x_1$ is increasing in $x_0$. Then, due to $B^2+AC>0$ from Claim 1 and $A>0$, the function $G_1$ is monotonically increasing on both $\hat E\cap E^+$ and $\hat E\cap E^-$. Accordingly, each of these intervals contains exactly one zero of $G_1$ (i.e., one initial value of the limit cycles). Taking Claim 1 into account, the two limit cycles are therefore located in $[0,1]\times E^+$ and $[0,1]\times E^-$, respectively.

    However, since $s_E=2$, we know by statement (i.2) of Theorem \ref{th2} that $f_1+f_2$ changes sign exactly once on each of $E^+$ and $E^-$. Applying statement (ii) of Theorem \ref{th2}, both limit cycles are hyperbolic, which contradicts their non-hyperbolicity. Hence, the estimate in Claim 2 is valid.
\vskip0.1cm

    {\em Cliam 3}: If equation \eqref{main equation1} has a limit cycle with multiplicity two in $[0,1]\times E$, then it has no hyperbolic limit cycles in the region.
\vskip0.1cm

    From Claim 2, we assume for a contradiction that the equation has two consecutive limit cycles in the region, denoted by $x=x(t;0,p_0)$ and $x=x(t;0,q_0)$, with the first one having multiplicity two and the second one being hyperbolic. According to the above argument, $p_0\in\hat E\cap(E^+\cup E^-)$ and $q_0\in\hat E$.
    Here we consider the case $p_0\in\hat E\cap E^+$ (i.e., $p_0>\frac{B}{A}$) for convenience, and the case $p_0\in\hat E\cap E^-$ can be treated in the same way.
    Then, from \eqref{eq14} and \eqref{eq15}, by using again the fact that two consecutive limit cycles must possess opposite stability on the adjacent side, we have
    \begin{align}\label{eq19}
    \text{sgn}\big(G(q_0)\big)=
    \left\{
  \begin{aligned}
         &-\text{sgn}(Ap_0-B)=-1,&  \text{if}\ q_0>p_0,\\
         &\text{sgn}(Ap_0-B)=1,&  \text{if}\ q_0<p_0.
    \end{aligned}
  \right.
    \end{align}

  On the other hand, recall that $G_1$ is monotonically increasing on $\hat E\cap E^+$. Since $G_1(p_0)=G(p_0)=0$ by \eqref{eq17} and \eqref{eq20}, one obtains
  \begin{align*}
   \text{sgn}\big(G(q_0)\big)
    =\text{sgn}\left(q_0-\frac{B}{A}\right)\cdot\text{sgn}\big(G_1(q_0)\big)
    =\left\{
  \begin{aligned}
         &1,&  \text{if}\ q_0>p_0,\\
         &-1,&  \text{if}\ \frac{B}{A}<q_0<p_0.
    \end{aligned}
  \right.
  \end{align*}
  This contradicts the case of $q_0>\frac{B}{A}$ in \eqref{eq19}. For the remaining case in \eqref{eq19}, i.e., $q_0\leq\frac{B}{A}$ with $G(q_0)>0$, it follows from \eqref{eq16} and Claim 1 that
    \begin{align*}
        (Aq_0-B)(Aq_1-B)=A G(q_0)+B^2+AC>0,
    \end{align*}
    where $q_1=x(\frac{1}{2};0,q_0)$.
    Therefore, $q_0,q_1\in E^-$, i.e., the limit cycle $x=x(t;0,q_0)$ is located in $[0,1]\times E^-$. Note that $x=x(t;0,p_0)$ is located in $[0,1]\times E^+$ by Claim 1 because $p_0\in\hat E\cap E^+$. A similar argument as in Claim 2 then yields the hyperbolicity of $x=x(t;0,p_0)$, which also shows a contradiction. Consequently, Claim 3 holds.
    \vskip0.1cm

    The arguments for Claims 2 and 3 can be intuitively understood by the relative position between the points $(x_0,x_1)$, which are induced by the limit cycles, and the hyperbola $Ax_{0}x_{1}-B(x_{0}+x_{1})-C=0$ on the $x_0x_1$-plane. See Figure \ref{fig-app2} for illustration.

    \begin{figure}[!htbp]
        \centering
            \includegraphics[scale=0.6]{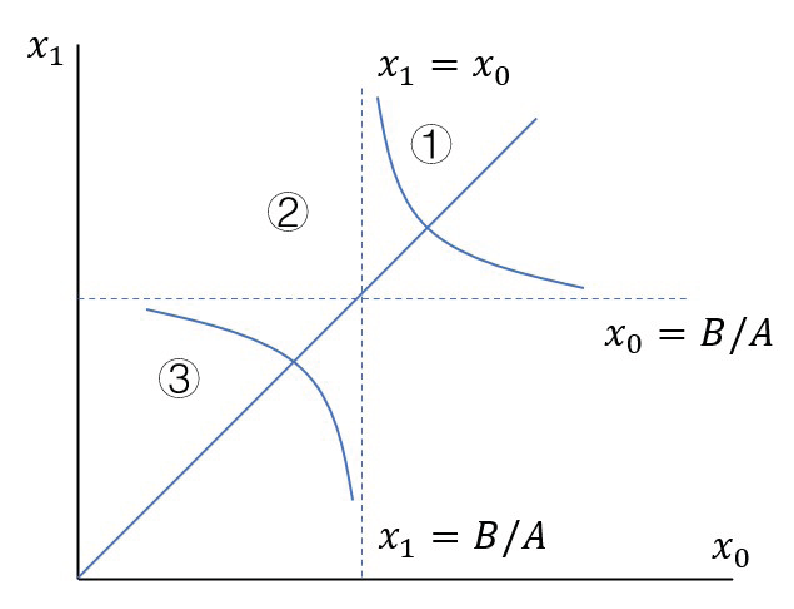}
        \caption{Regional division on stability and multiplicity of limit cycles in the subcase $A\neq0$ with $f_1|_E>0$. The hyperbola and the regions \normalsize{\textcircled{\scriptsize{1}}}\normalsize, \normalsize{\textcircled{\scriptsize{2}}}\normalsize \ and \normalsize{\textcircled{\scriptsize{3}}}\normalsize\ divided by the hyperbola and $x_1=x_0$ correspond to different stabilities of limit cycles.}\label{fig-app2}
    \end{figure}

    Now let us present the last claim.
    \vskip0.1cm

    {\em Cliam 4}: Equation \eqref{main equation1} has at most two hyperbolic limit cycles in $[0,1]\times E$.
\vskip0.1cm

    Our argument is restricted to the case $E\subseteq\mathbb R^+$ because the conclusion for the case $E\subseteq\mathbb R^-$ follows in exactly the same way.  So far, for all the cases listed in Table \ref{table5}, we have actually known by the previous argument, from Proposition \ref{prop1} to the above Claim 3, that equation \eqref{main equation1} has at most three limit cycles (counted with multiplicities) provided that one of them is non-hyperbolic. In order to verify this last claim, we assume for a contradiction that there exist four consecutive hyperbolic limit cycles of the equation, such that the first one is $x=0$ and the latter three are located in $[0,1]\times E$. Similar to the proof in subcase 1, we consider equation \eqref{main equation1} as a family of rotated equations on $[0,1]\times\mathbb R^+$ with respect to the parameter $c_1$. Then, we can denote these three limit cycles by $x=\varphi_{1}(t;c_1)$, $x=\varphi_{2}(t;c_1)$ and $x=\varphi_{3}(t;c_1)$, with $0<\varphi_{1}(t;c_1)<\varphi_{2}(t;c_1)<\varphi_{3}(t;c_1)$ and $c_1\in J$, the maximal interval where all $\varphi_{i}$ are well-defined and keep their hyperbolicity.

    According to Proposition \ref{proposition2.6}, the limit cycle $x=0$ is hyperbolic if and only if $c_1\neq-c_2$. Taking the change of variable $t\mapsto -t$ in the equation if necessary, we can suppose without loss of generality that $c_1>-c_2$ for $c_1\in J$, i.e., $J\subseteq (-c_2,+\infty)$. Thus, $x=0$ and $x=\varphi_{2}(t;c_1)$ are unstable, whereas $x=\varphi_{1}(t;c_1)$ and $x=\varphi_{3}(t;c_1)$ are stable. Note that $\frac{\partial f_1}{\partial c_1}=x>0$ and $\frac{\partial f_2}{\partial c_1}=0$ for $(t,x,c_1)\in[0,1]\times\mathbb R^+\times J$. From Lemma \ref{property of rotated}, both $x=\varphi_{1}(t;c_1)$ and $x=\varphi_{3}(t;c_1)$ decrease and the middle one $x=\varphi_{2}(t;c_1)$ increases as $c_1$ decreases in $J$. However, this together with the limit cycle $x=0$, implies the existence of at least four limit cycles (counted with multiplicities) of equation \eqref{main equation1} when $c_1$ reaches the infimum of $J$, with at least one limit cycle being non-hyperbolic. We arrive at a contradiction and thus Claim 4 is proved.
    \vskip0.1cm


    Finally, on account of Claim 2, Claim 3, Claim 4 and Lemma \ref{prop5}, equation \eqref{main equation1} has at most two limit cycles (counted with multiplicities) in $[0,1]\times E$ for subcase $A\neq0$. This ends the proof of the proposition.
\end{proof}

\begin{proof}[Proof of Theorem \ref{main theorem}]
    As the argument is reduced to that for equation \eqref{main equation1} with $f_1+f_2\not\equiv0$, the upper bound for the number of limit cycles in the assertion is immediately obtained from Proposition \ref{prop1} and Proposition \ref{prop2}. Moreover, this upper bound can be achieved when $f_1=f_2$ has exactly three zeros. In such a case equation \eqref{main equation1} has three constant limit cycles.
\end{proof}

\begin{remark}
  The analysis in application 2, constitutes a concrete example of how the zeros of the function $f_1+f_2$ simultaneously affect the distributions of the limit cycles, both the constant and the non-constant ones, of the equation of the form \eqref{equation2}.
\end{remark}

\subsection{Application 3}
In this final application, we focus on model \eqref{mosquito} under two strategies: $T>\overline T$ and $T<\overline T$. We will prove Proposition \ref{subapplication2.1} and Theorem \ref{subapplication2.2}, respectively, again following our established procedure.


\subsubsection{Strategy $T>\overline T$}
As introduced in Section 1, the function $h(t)$ in equation \eqref{mosquito} is given by \eqref{eq3}. Using Lemma \ref{lemma1}, the equation in one period (with $k=0$) can be normalized to
\begin{equation}\label{YU11}
    \frac{dw}{dt}
     =\left\{
  \begin{aligned}
         & h_{1}(w)=2\overline{T}\left(\frac{aw}{w+c}-\big(\mu+\xi (w+c)\big)\right)w,  &  t\in \Big[0,\frac{1}{2}\Big),\ w\in \mathbb{R}_{0}^{+},\\
         & h_{2}(w)=-2\xi\left(T-\overline{T}\right) (w-A)w,  &  t\in \Big[\frac{1}{2},1\Big],\ w\in \mathbb{R}_{0}^{+}.
    \end{aligned}
  \right.
\end{equation}
where $\xi$, $c>0$, $a>\mu>0$ and $A=\frac{a-\mu}{\xi}$.

Next, let us recall several important notations and thresholds from \cite{YULI,5}:
\begin{equation}\label{thresholds1}
    g^{*}:=\frac{(a-\mu)^{2}}{4\xi a},\ c^{*}:=\frac{-a+\sqrt{a^{2}+4a(a-\mu)}}{2\xi}\text{ and }\ T^{*}:=\frac{a+\xi c}{a-\mu}\overline{T},
\end{equation}
where it has been shown in the articles that $g^*<c^*$.
We first pre-analyze equation \eqref{YU11} in two steps.
\vskip0.2cm

\textbf{Step 1:} Determine the common zeros of $h_{1}$ and $h_{2}$, and obtain $\mathcal U=\mathbb R_0^+\backslash\{w|h_1(w)h_2(w)=0\}$.

Clearly, $h_1$ can be rewritten as $h_1(w)=\frac{2\overline T w}{w+c}\cdot \overline h_1(w)$, where
$$\overline h_1(w)=-\xi w^2+(a-\mu-2\xi c)w-(\xi c+u)c.$$
Then, a straightforward analysis of $\overline h_1$ and $h_2$ shows the following facts (see Figure \ref{Fig-app3.1} for illustration):
\begin{itemize}
    \item If $0<c< g^{*}$, then $h_{1}$ has exactly three zeros, one at $w=0$ and the other two, denoted by $w=\lambda_{1}(c)$ and $w=\lambda_{2}(c)$, satisfying $0<\lambda_{1}(c)< \lambda_{2}(c)<A$. Furthermore, $h_{1}$ is negative on $(0,\lambda_{1})\cup(\lambda_{2},+\infty)$ and positive on $(\lambda_{1},\lambda_{2})$, respectively.
    \item If $c=g^{*}$, then $h_{1}$ has exactly two zeros, one at $w=0$ and the other, denoted by $w=\lambda_*$, satisfying $0<\lambda_*<A$. Furthermore, $h_{1}\leq 0$ on $(0,+\infty)$.
    \item If $c>g^{*}$, then $h_{1}$ has only one zero $w=0$ and is negative on $(0,+\infty)$.
    \item $h_{2}$ has exactly two zeros $w=0$ and $w=A$.
        Furthermore, $h_{2}$ is positive on $(0,A)$ and negative on $(A,+\infty)$, respectively.
\end{itemize}
Accordingly, there are no common zeros of $h_{1}$ and $h_{2}$ except $w=0$, i.e. equation \eqref{YU11} has only one constant limit cycle $w=0$.
And we have that
\begin{itemize}
    \item[(a)] If $0<c< g^{*}$, then $\mathcal U=(0,\lambda_{1})\cup(\lambda_{1},\lambda_{2})\cup(\lambda_{2},A)\cup(A,+\infty)$.
    \item[(b)] If $c=g^{*}$, then $\mathcal U=(0,\lambda_*)\cup(\lambda_*,A)\cup(A,+\infty)$.
    \item[(c)] If $c>g^{*}$, then $\mathcal U=(0,A)\cup(A,+\infty)$.
  \end{itemize}

\textbf{Step 2:} Determine the sign changes of $h_{1}+h_{2}$.

From \eqref{YU11} and \eqref{thresholds1}, we write $h_1+h_2$ as
    \begin{align}\label{eq21}
    h_1(w)+h_2(w)=
    2\left(\frac{a\overline{T}}{w+c}-\xi T\right)w^2+2(a-\mu)\left(T-T^*\right)w.
    \end{align}
For the sake of brevity and compactness, we directly present here the following facts of this function (see again Figure \ref{Fig-app3.1} for illustration). The analysis of obtaining these facts, while not technically difficult, is a little tedious and is arranged in Appendix \ref{statements of application3.1}.
    \begin{enumerate}
        \item[(d)]If $\overline{T}<T<T^{*}$ and $0<c\leq g^{*}$, then $h_{1}+h_{2}$ has exactly two positive zeros, denoted by $w=\mu_{1}$ and $w=\mu_{2}$, satisfying
        \begin{enumerate}
            \item[(d.1)]$\mu_{1}\in (0,\lambda_{1})$ and $\mu_{2}\in (\lambda_{2},A)$, if additionally $0<c< g^{*}$.
            \item[(d.2)]$\mu_{1}\in (0,\lambda_*)$ and $\mu_{2}\in (\lambda_*,A)$, if additionally $c= g^{*}$.
        \end{enumerate}
        Furthermore, $h_{1}+h_{2}$ is positive on $(\mu_{1},\mu_{2})$ and negative on $(0,\mu_{1})\cup(\mu_{2},+\infty)$, respectively.
        \item[(e)] If $\overline{T}<T<T^{*}$ and $g^{*}<c<c^{*}$, then $h_{1}+h_{2}$ has at most two positive zeros on $(0,A)$ and no zeros on $(A,+\infty)$.
        \item[(f)] If $\overline{T}<T\leq T^{*}$ and $c\geq c^{*}$, then $h_{1}+h_{2}\leq 0$.
        \item[(g)] If either $T>T^{*}$, or $T= T^{*}$ and $0<c<c^{*}$, then $h_{1}+h_{2}$ has only one positive zero, denoted by $w=\mu_*(c)$, satisfying $\mu_*\in (0,A)$.
        Furthermore, $h_{1}+h_{2}$ is positive on $(0,\mu_*)$ and negative on $(\mu_*,+\infty)$.
    \end{enumerate}
Hence, on account of statements (a)--(f), the number of sign changes $s_{(\cdot,\cdot)}$ of $h_{1}+h_{2}$ on the intervals related to $\mathcal{U}$, can be summarized in Table \ref{the table of application3.1} below.
\begin{table}[!ht]
    \center
    \begin{threeparttable}
    \begin{tabular}{|c|c|c|c|}
        \hline
         \diagbox[innerwidth=2cm]{$T$}{$c$} &$0<c\leq g^{*}$&$g^{*}<c<c^{*}$&$c\geq c^{*}$\\
        \hline
       $\overline{T}<T<T^{*}$&\makecell[c]{\rm{$s_{(0,\lambda_{1})}=s_{(\lambda_{2},A)}=1$ and}\\ \rm{$s_{(\lambda_{1},\lambda_{2})}$=$s_{(A,+\infty)}=0$}}&\makecell[c]{\rm{$s_{(0,A)}\leq 2$ }\\ \rm{and $s_{(A,+\infty)}=0$}}&\multicolumn{1}{c|}{\multirow{2}{*}{\rm{$s_{(0,A)}=s_{(A,+\infty)}=0$}}} \\
       \cline{1-3}
       $T=T^{*}$& \multicolumn{2}{c|}{\multirow{2}{*}{\makecell[c]{\rm{$s_{(0,A)}=1$ and $s_{(A,+\infty)}=0$}}}}&\\
       \cline{1-1} \cline{4-4}
       $T>T^{*}$& \multicolumn{3}{c|}{} \\
       \hline
    \end{tabular}
    \end{threeparttable}
    \caption{The number of sign changes of $h_{1}+h_{2}$. In the case where $\overline{T}<T<T^{*}$ and $c=g^{*}$, we use the convention $\lambda_1=\lambda_2=\lambda_*$. }\label{the table of application3.1}
    \vspace{-0.6cm}
\end{table}

\begin{figure}[!htbp]
    \centering
    \begin{subfigure}[b]{0.49\textwidth}
        \includegraphics[scale=0.32]{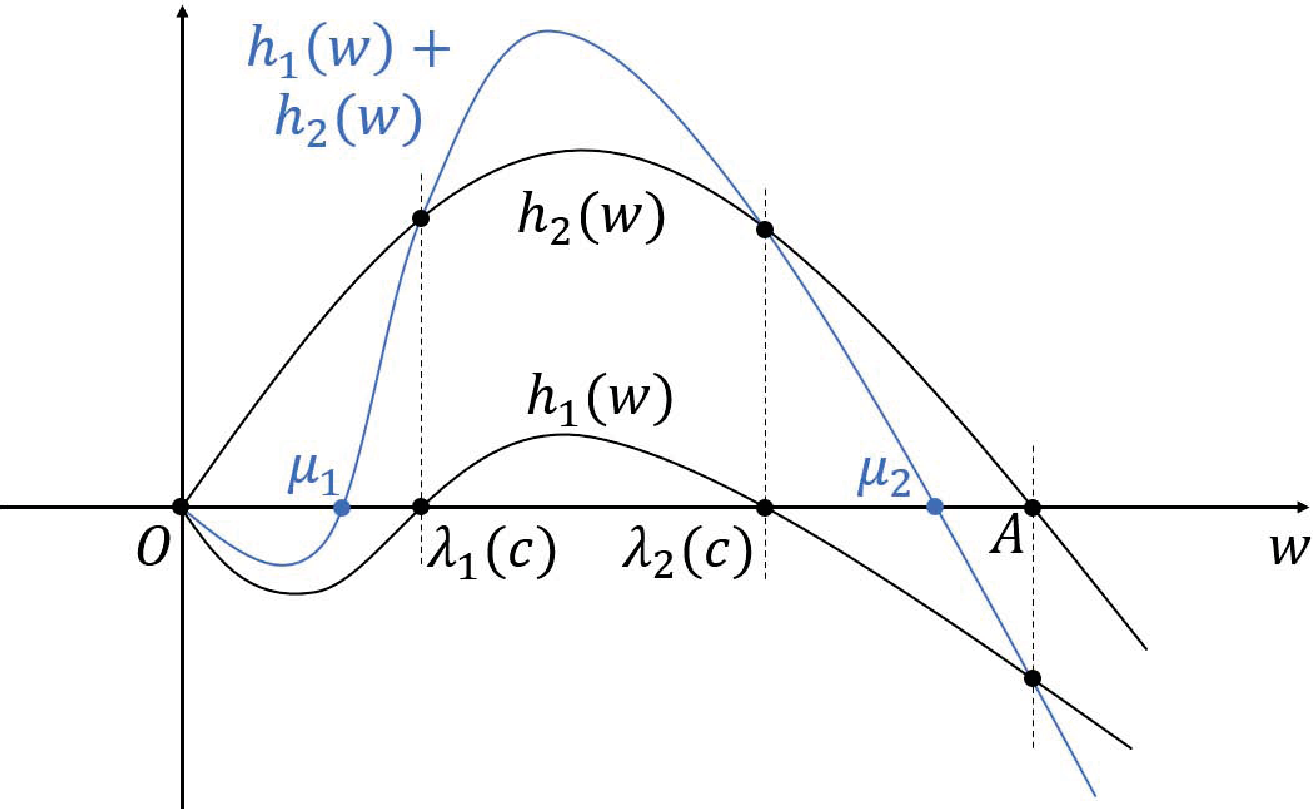}
        \caption{$\overline T<T<T^*$ and $0<c\leq g^*$}
    \end{subfigure}
    \begin{subfigure}[b]{0.49\textwidth}
        \includegraphics[scale=0.33]{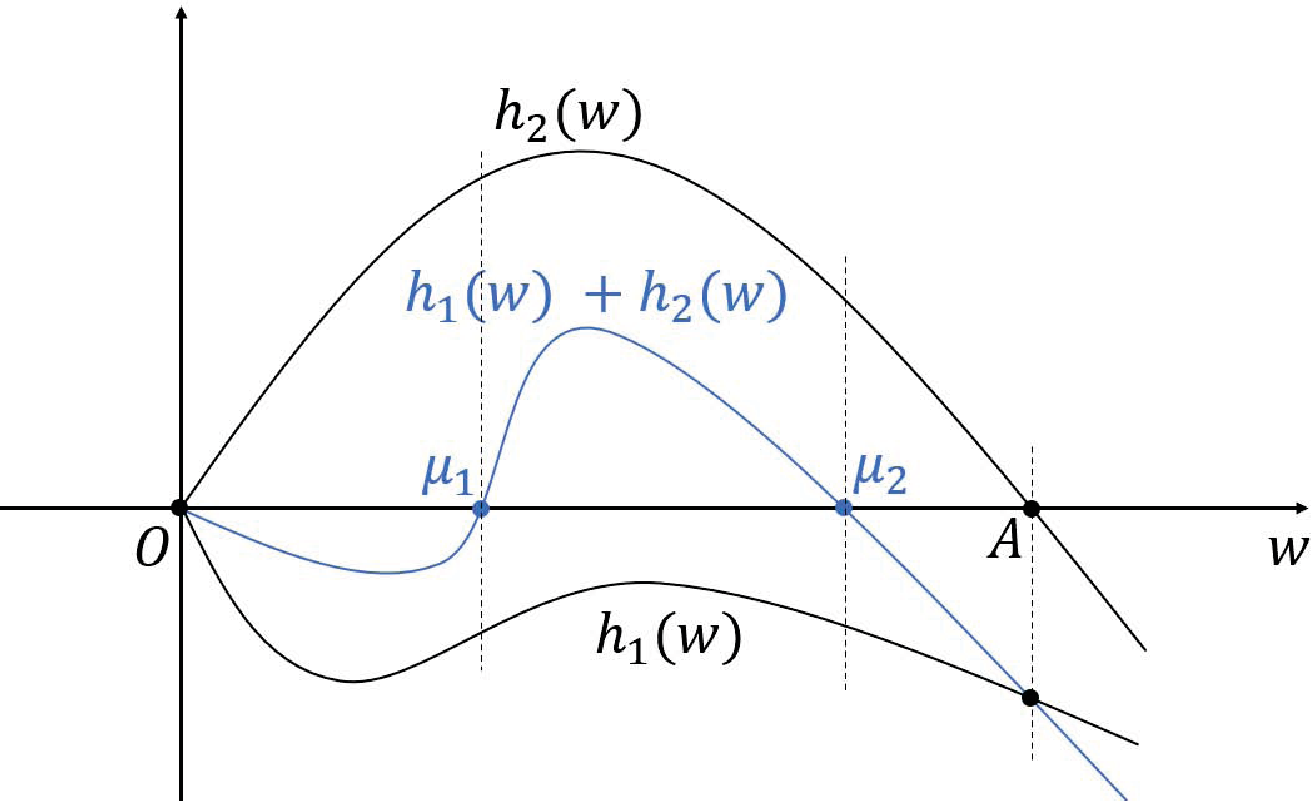}
        \caption{$\overline T<T<T^*$ and $g^*<c<c^*$}
    \end{subfigure}
    \begin{subfigure}[b]{0.49\textwidth}
        \includegraphics[scale=0.32]{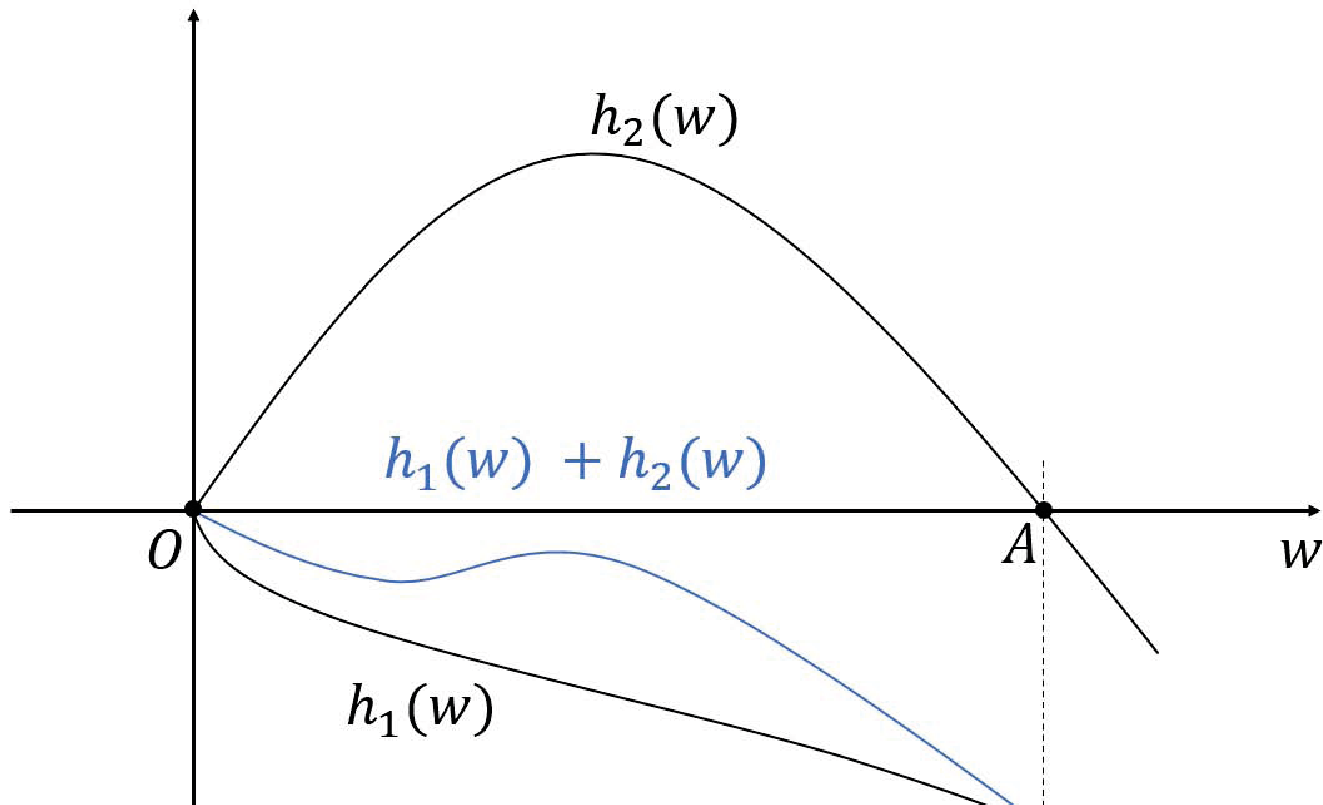}
        \caption{$\overline T<T\leq T^*$ and $c\geq c^*$}
    \end{subfigure}
    \begin{subfigure}[b]{0.49\textwidth}
        \includegraphics[scale=0.61]{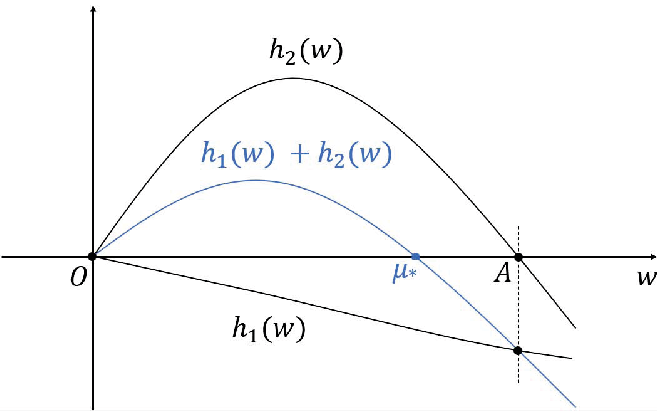}
        \caption{$\overline T>T^*$, or $T=T^*$ and $0<c<c^*$}
    \end{subfigure}
    \caption{The relative positions of the zeros of $h_1$, $h_2$, and $h_1+h_2$ for different values of $T$ and $c$.}\label{Fig-app3.1}
\end{figure}

We shall now give the following result according to Table \ref{the table of application3.1}, Theorem \ref{th2}, Propositions \ref{propostion2.5} and \ref{proposition2.6}.

\begin{proposition}\label{proposition 3.8}
    The following statements hold.
    \begin{itemize}
        \item[(i)] If $\overline{T}<T<T^{*}$ and $0<c\leq g^{*}$, then equation \eqref{YU11} has exactly two non-constant limit cycles, where the lower one is unstable and the upper one is stable. Moreover, the constant limit cycle $w=0$ is stable.
        \item[(ii)] If $\overline{T}<T<T^{*}$ and $g^{*}<c<c^{*}$, then equation \eqref{YU11} has at most two non-constant limit cycles. Moreover, the constant limit cycle $w=0$ is stable.
        \item[(iii)] If $\overline{T}<T\leq T^{*}$ and $c\geq c^{*}$, then equation \eqref{YU11} has no non-constant limit cycles. Moreover, the constant limit cycle $w=0$ is upper-stable.
        \item[(iv)] If either $T>T^{*}$, or $T= T^{*}$ and $0<c<c^{*}$, then equation \eqref{YU11} has only one non-constant limit cycle, which is stable. Moreover, the constant limit cycle $w=0$ is upper-unstable.
    \end{itemize}
\end{proposition}
\begin{proof}
    We begin by showing the stability of $w=0$.
    Let $P(w_{0})=w(1;0,w_{0})$ be the Poincar\'{e} map of equation \eqref{YU11}, where $w(t;0,w_{0})$ is the solution of the equation satisfying $w(0;0,w_{0})=w_{0}$.
    From Proposition \ref{proposition2.6} and \eqref{eq21}, one has
    $
        P'(0)=\exp\big((a-\mu)(T-T^{*})\big).
    $
    Thus, $w=0$ is unstable (resp. stable) if $T>T^{*}$ (resp. $T<T^{*}$).
    If $T=T^{*}$, then $P'(0)=1$, and it also follows from Proposition \ref{proposition2.6} and \eqref{eq21} that
    \begin{align*}
        P''(0)
        &=4\left(\frac{a\overline{T}}{c}-\xi T^{*}\right)\int_{0}^{\frac{1}{2}}\exp \big(h'_1(0)t\big)dt\\
        &=-\frac{4\overline T}{(a-\mu)c}\int_{0}^{\frac{1}{2}}\exp \big(h'_1(0)t\big)dt\cdot\Big((\xi c)^{2}+a(\xi c)-a(a-\mu)\Big).
    \end{align*}
    Note that the threshold $c^*$ satisfies $(\xi c^*)^{2}+a(\xi c^*)-a(a-\mu)=0$. Therefore, $w=0$ is upper-unstable (resp. upper-stable) if additionally $0<c<c^{*}$ (resp. $c>c^{*}$). Since Proposition \ref{proposition2.6} and \eqref{eq21} further gives $P'''(0)<0$ when $P'(0)-1=P''(0)=0$, we obtain that $w=0$ is stable for the last case $T=T^*$ and $c=c^*$.
    Consequently, the assertions in the statements for $w=0$ all follow.

    Next we prove the remaining parts of the statements one by one.

    (i) By virtue of Table \ref{the table of application3.1} and Theorem \ref{th2}, it is clear that equation \eqref{YU11} has at most two non-constant limit cycles.
    Furthermore, the argument in Step 1 shows that: $h_{1}(\lambda_{1})=0$ and $h_{2}(\lambda_{1})>0$; $h_{1}(\lambda_{2})=0$ and $h_{2}(\lambda_{2})>0$; $h_{1}(A)<0$ and $h_{2}(A)=0$. These together with the stability of $w=0$ when $T<T^{*}$, yield an unstable limit cycle in $[0,1]\times(0,\lambda_{1})$ and a stable limit cycle in $[0,1]\times(\lambda_{2},A)$ of the equation.
    Hence, statement (i) holds.

    (ii) One can easily check that
    $h_{1}'''(w)=\frac{12 ac\overline{T}}{(w+c)^{4}}>0$
    and
    $h_{2}'''(w)\equiv0$.
    According to Proposition \ref{propostion2.5}, equation \eqref{YU11} has at most three limit cycles, including $w=0$.
    Thus, the number of non-constant limit cycles does not exceed two. Statement (ii) is valid.

    (iii) The assertion follows immediately from Table \ref{the table of application3.1} and statement (i) of Theorem \ref{th2}.

    (iv) Again due to Table \ref{the table of application3.1} and statement (ii) of Theorem \ref{th2}, equation \eqref{YU11} has at most one non-constant limit cycle.
    Since $h_{1}(A)<0$ and $h_{2}(A)=0$, and $w=0$ is upper-unstable in this case, there exactly exists a stable limit cycle in $[0,1]\times (0,A)$.
    This completes the proof of statement (iv).
\end{proof}

\begin{proof}[Proof of Proposition \ref{subapplication2.1}]
    Recall that equation \eqref{mosquito} is reduced to equation \eqref{YU11} under assumption. The conclusion can be immediately obtained by Proposition \ref{proposition 3.8}.
\end{proof}

\begin{remark}\label{rem}
We have the next two comments for equation \eqref{YU11} and Proposition \ref{subapplication2.1}.

\begin{enumerate}
[\text{(i)}] \item It is worth noting that, \eqref{YU11} is essentially an equation of the Abel type that is considered in Application 2. In fact, under the Cherkas' transformation $x=\frac{\omega}{\omega+c}$ introduced in Section 1, it becomes the following equation of the form \eqref{main equation}:
\begin{align*}
  \frac{dx}{dt}
    =\left\{
  \begin{aligned}
       & -2\overline T\left(ax^2-(a+\mu)x+\mu+\xi c\right)x,&  t\in \Big[0,\frac{1}{2} \Big),\ x\in [0,1),\\
       & -2\xi\left(T-\overline T\right)\big((A+c)x-A\big)x,&  t\in \Big[\frac{1}{2},1\Big],\ x\in [0,1).
    \end{aligned}
  \right.
\end{align*}
\end{enumerate}
\begin{enumerate}
[\text{(ii)}]
\item There is a small improvement to Proposition \ref{subapplication2.1}, utilizing the theory of rotated equations. Since $\frac{\partial h_{1}}{\partial c}=-2\overline{T}w\Big(\frac{aw}{(w+c)^{2}}+\xi\Big)<0$
and $\frac{\partial h_{2}}{\partial c}=0$ for $(t,w,c)\in[0,1]\times\mathbb R^+\times\mathbb R^+$,
by Definition \ref{definition of rotated}, equation \eqref{YU11} forms a family of rotated equations with respect to $c$.
Applying Lemma \ref{han} and Proposition \ref{proposition 3.8},
for each given $T\in(\overline{T},T^{*})$, there exists $c^{\diamond}\in (g^{*},c^{*})$ such that
\begin{itemize}
    \item Equation \eqref{YU11} has exactly two non-constant limit cycles if $0<c<c^{\diamond}$, where the lower one is unstable and the upper one is stable.
    \item Equation \eqref{YU11} has only one non-constant limit cycle if $c=c^{\diamond}$, which is upper-stable and lower-unstable.
    \item Equation \eqref{YU11} has no non-constant limit cycles if $c>c^{\diamond}$.
\end{itemize}
\end{enumerate}
\end{remark}

\subsubsection{Strategy $T<\overline T$}
Under this strategy, the function $h(t)$ in equation \eqref{mosquito} is given by \eqref{eq4}. Then, as previously applied, by Lemma \ref{lemma1} we can normalize the equation in one period (with $k=0$) to
\begin{equation}\label{25}
    \frac{dw}{dt}
    =\left\{
  \begin{aligned}
       &h_{1}(w)= 2q\left(\frac{aw}{w+(p+1)c}-\big(\mu+\xi (w+(p+1)c)\big)\right)w,&  t\in \Big[0,\frac{1}{2} \Big), w\in \mathbb{R}_{0}^{+},\\
       &h_{2}(w)= 2(T-q)\left(\frac{aw}{w+pc}-\big(\mu+\xi (w+pc)\big)\right)w,&  t\in \Big[\frac{1}{2},1\Big], w\in \mathbb{R}_{0}^{+},
    \end{aligned}
  \right.
\end{equation}
where $\xi, c >0$, $a>\mu>0$, $T>q>0$ and $p\in \mathbb{N}^{+}$.

To continue the argument, we recall two thresholds introduced in \cite{BOYU}:
\begin{equation}\label{thresholds2}
    g_{1}^{*}=\frac{(a-\mu)^{2}}{4a(p+1)\xi},\indent\  g_{2}^{*}=\frac{(a-\mu)^{2}}{4ap\xi}.
\end{equation}
In the following we pre-analyze equation \eqref{25} in two steps, in terms of these thresholds and a new one.

\textbf{Step 1:} Determine the common zeros of $h_{1}$ and $h_{2}$, and obtain $\mathcal{U}=\mathbb{R}_{0}^{+}\backslash\{w|h_1(w)h_2(w)=0\}$.

For $i=1,2$, let
\begin{align*}
\hat h_i(w)=-\xi\big(w+(p+2-i)c\big)^2-\mu\big(w+(p+2-i)c\big)+aw.
\end{align*}
Then, one has $h_1(w)=\frac{2qw}{w+(p+1)c}\cdot\hat h_1(w)$ and $h_2(w)=\frac{2(T-q)w}{w+pc}\cdot\hat h_2(w)$.
A straightforward analysis for $\hat h_1$ and $\hat h_2$ easily gives the following facts (see Figure \ref{Figures of application3.2} for illustration):

\begin{itemize}
    \item $w=0$ is a common zero of $h_{1}$ and $h_{2}$.
    \item If $c>g_{i}^{*}$, then $h_{i}$ has no positive zeros and is negative on $(0,+\infty)$, $i=1,2$.
    \item If $c=g_{i}^{*}$, then $h_{i}$ has only one positive zero $w=\lambda_*:=\frac{a^{2}-\mu^{2}}{4 a\xi}$, $i=1,2$. Furthermore, $h_{i}\leq 0$ on $(0,+\infty)$.
    \item If $0<c< g_{i}^{*}$, then $h_{i}$ has exactly two positive zeros, denoted by $w=\lambda_{i,1}(c)$ and $w=\lambda_{i,2}(c)$, with   $\lambda_{i,1}(c)<\lambda_*<\lambda_{i,2}(c)$ and $i=1,2$. In particular,
        $$\lambda_{2,1}(c)<\lambda_{1,1}(c)<
        \lambda_*<\lambda_{1,2}(c)<\lambda_{2,2}(c)\ \ \text{when}\ \ 0<c<g_1^*.$$
    (The above inequalities come from the fact that each $\hat h_i$ is decreasing in the parameter $c$ and $\hat h_1<\hat h_2$ on $(0,+\infty)$).
    Also, $h_{i}$ is negative on $(0,\lambda_{i,1})\cup (\lambda_{i,2},+\infty)$ and positive on $(\lambda_{i,1},\lambda_{i,2})$.
\end{itemize}
Accordingly, there are no common zeros of $h_{1}$ and $h_{2}$ except $w=0$, i.e. equation \eqref{25} has only one constant limit cycle $w=0$.
And we have that
\begin{itemize}
    \item[(a)] If $0<c< g_{1}^{*}$, then $\mathcal U=\left(0,\lambda_{2,1}\right)\cup(\lambda_{2,1},\lambda_{1,1})
        \cup(\lambda_{1,1},\lambda_{1,2})\cup(\lambda_{1,2},\lambda_{2,2})
        \cup(\lambda_{2,2},+\infty)$.
    \item[(b)] If $c=g_{1}^{*}$, then $\mathcal U=\left(0,\lambda_{2,1}\right)\cup(\lambda_{2,1},\lambda_*)\cup (\lambda_*,\lambda_{2,2})\cup(\lambda_{2,2},+\infty)$.
    \item[(c)] If $g_{1}^{*}<c<g_{2}^{*}$, then $\mathcal  U=\left(0,\lambda_{2,1}\right)\cup(\lambda_{2,1},\lambda_{2,2})\cup(\lambda_{2,2},+\infty)$.
    \item[(d)] If $c=g_{2}^{*}$, then $\mathcal U=\left(0,\lambda_{*}\right)\cup(\lambda_{*},+\infty)$.
    \item[(e)] If $c>g_{2}^{*}$, then $\mathcal U=(0,+\infty)$.
  \end{itemize}

\textbf{Step 2: }Determine the sign changes of $h_{1}+h_{2}$.

Here we introduce a new threshold
        \begin{equation}\label{threshold3}
            T^{***}
            :=\left\{
  \begin{aligned}
            &0, &\text{if}\ 0<c\leq g_1^* \text{ or } c\geq g_2^*,\\
            &\inf_{w\in (\lambda_{2,1},\lambda_{2,2})}
            \left(q\left(1-\frac{w+pc}{w+(p+1)c}\cdot\frac{\hat h_1(w)}{\hat h_2(w)}\right)\right),
            &\text{if}\ g_1^*<c<g_2^*.
  \end{aligned}
  \right.
        \end{equation}
We emphasize that $T^{***}$ is well-defined when $g_1^*<c<g_2^*$. Indeed, since $\hat h_i$ has the same sign as $h_i$ on $(0,+\infty)$ for $i=1,2$, the facts stated in Step 1 shows that $-\frac{\hat h_1}{\hat h_2}$ is positive on $(\lambda_{2,1},\lambda_{2,2})$ in this case, which ensures the existence of the infimum in \eqref{threshold3}.

Similar to the arrangement of the analysis for equation \eqref{YU11}, we present the next three facts directly for $h_1+h_2$
(see again Figure \ref{Figures of application3.2} for illustration), and provide the detailed analysis in Appendix \ref{statements of application3.2}.

\begin{enumerate}
    \item[(f)] If $0<c\leq g_{1}^{*}$, then $h_{1}+h_{2}$ has exactly two positive zeros, denoted by $x=\mu_{1}$ and $x=\mu_{2}$, satisfying $\mu_{1}\in(\lambda_{2,1},\lambda_{1,1})$ and $\mu_{2}\in(\lambda_{1,2},\lambda_{2,2})$.
                Furthermore, $h_{1}+h_{2}$ is negative on $(0,\mu_{1})\cup(\mu_{2},+\infty)$ and positive on $(\mu_{1},\mu_{2})$, respectively.
    \item[(g)] If $g_{1}^{*}<c< g_{2}^{*}$, then
        \begin{enumerate}
            \item[(g.1)] $h_{1}+h_{2}\leq 0$, if additionally $T\leq T^{***}$.
            \item [(g.2)] $h_{1}+h_{2}$ has at most two positive zeros on $(\lambda_{2,1},\lambda_{2,2})$ and is negative on $(0,\lambda_{2,1})\cup(\lambda_{2,2},+\infty)$, if additionally $T^{***}<T<\overline{T}$.
        \end{enumerate}
    \item[(h)] If $c\geq g_{2}^{*}$, then $h_{1}+h_{2}$ has no positive zeros and is negative on $(0,+\infty)$.
\end{enumerate}
Therefore, the combination of statements (a)--(h), yields the number of sign changes of $h_{1}+h_{2}$ on the connected components of $\mathcal{U}$. See Table \ref{Table of application3.2} below.
\begin{table}[!ht]
    \center
    \begin{threeparttable}
    \begin{tabular}{|c|c|c|c|}
        \hline
         \diagbox[innerwidth=2cm]{$T$}{$c$} &$0<c\leq g_{1}^{*}$& $g_{1}^{*}<c<g_{2}^{*}$&$c\geq g_{2}^{*}$\\
        \hline
       $0<T<T^{***}$& \multicolumn{1}{c|}{\multirow{3}{*}{\makecell[c]{ \rm{$s_{(\lambda_{2,1},\lambda_{1,1})}=s_{(\lambda_{1,2},\lambda_{2,2})}$}\\ \rm{$=1$ and }\\ \rm{$s_{(0,\lambda_{2,1})}=s_{(\lambda_{1,1},\lambda_{1,2})}$}\\\rm{$=s_{(\lambda_{2,2},+\infty)}=0$}}}}  &\multicolumn{1}{c|}{\makecell[c]{{\rm{$s_{(0,\lambda_{2,1})}=s_{(\lambda_{2,1},\lambda_{2,2})}$}}\\ \rm{$=s_{(\lambda_{2,2},+\infty)}=0$}}}&\multicolumn{1}{c|}{\multirow{3}{*}{\rm{$s_{(0,+\infty)}=0$}}} \\
       \cline{1-1}
       $T=T^{***}$& & & \\
       \cline{1-1}\cline{3-3}
       $T^{***}<T<\overline{T}$& & \makecell[c]{{\rm{$s_{(\lambda_{2,1},\lambda_{2,2})}\leq 2$} and }\\{\rm{ $s_{(0,\lambda_{2,1})}=s_{(\lambda_{2,2},+\infty)}=0$}}}& \\
       \hline
    \end{tabular}
    \end{threeparttable}
    \caption{The number of sign changes of $h_{1}+h_{2}$.
    In the case where $c=g_{1}^{*}$, we use the convention $\lambda_{1,1}=\lambda_{1,2}=\lambda_*$.}\label{Table of application3.2}
    \vspace{-0.6cm}
\end{table}

\begin{figure}[!htbp]
    \centering
    \begin{subfigure}[b]{0.49\textwidth}
        \includegraphics[scale=0.32]{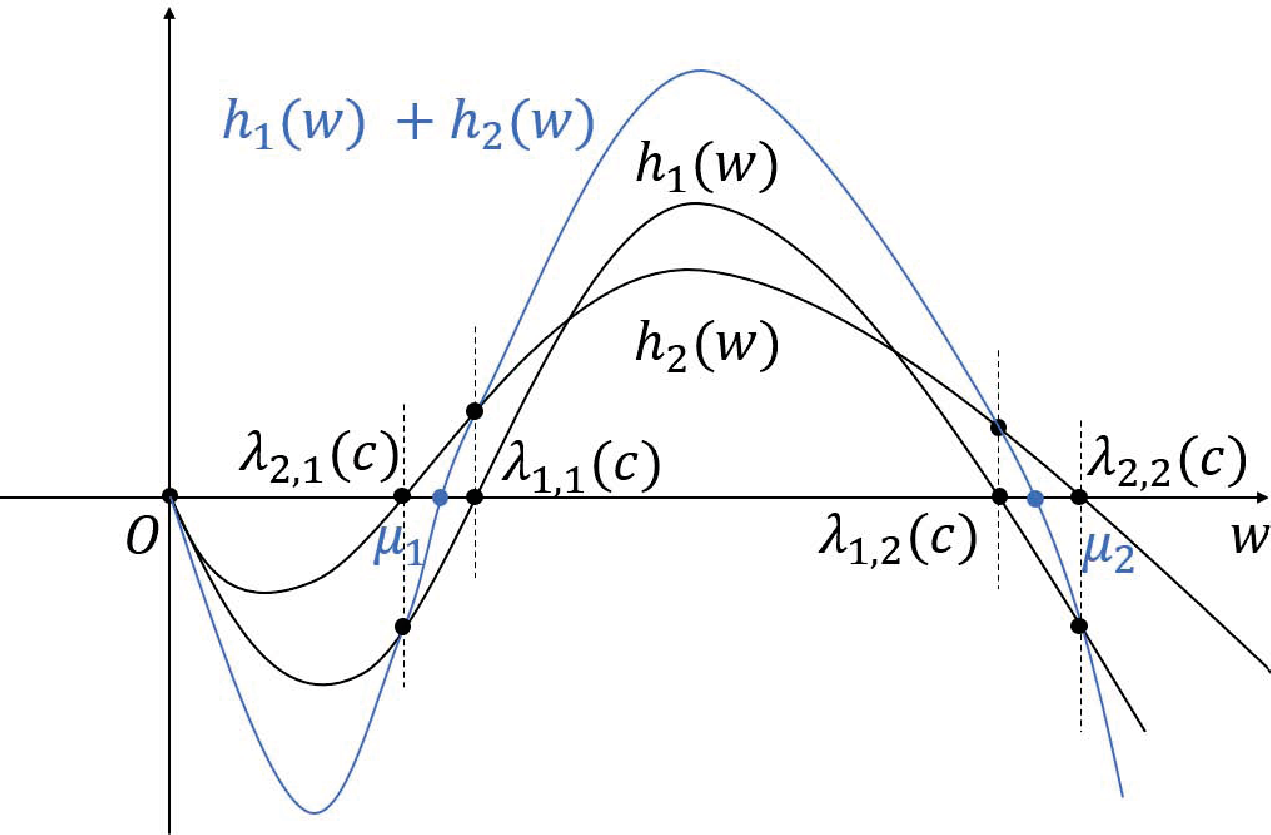}
        \caption{$0<c\leq g_1^*$}
    \end{subfigure}
    \begin{subfigure}[b]{0.49\textwidth}
        \includegraphics[scale=0.32]{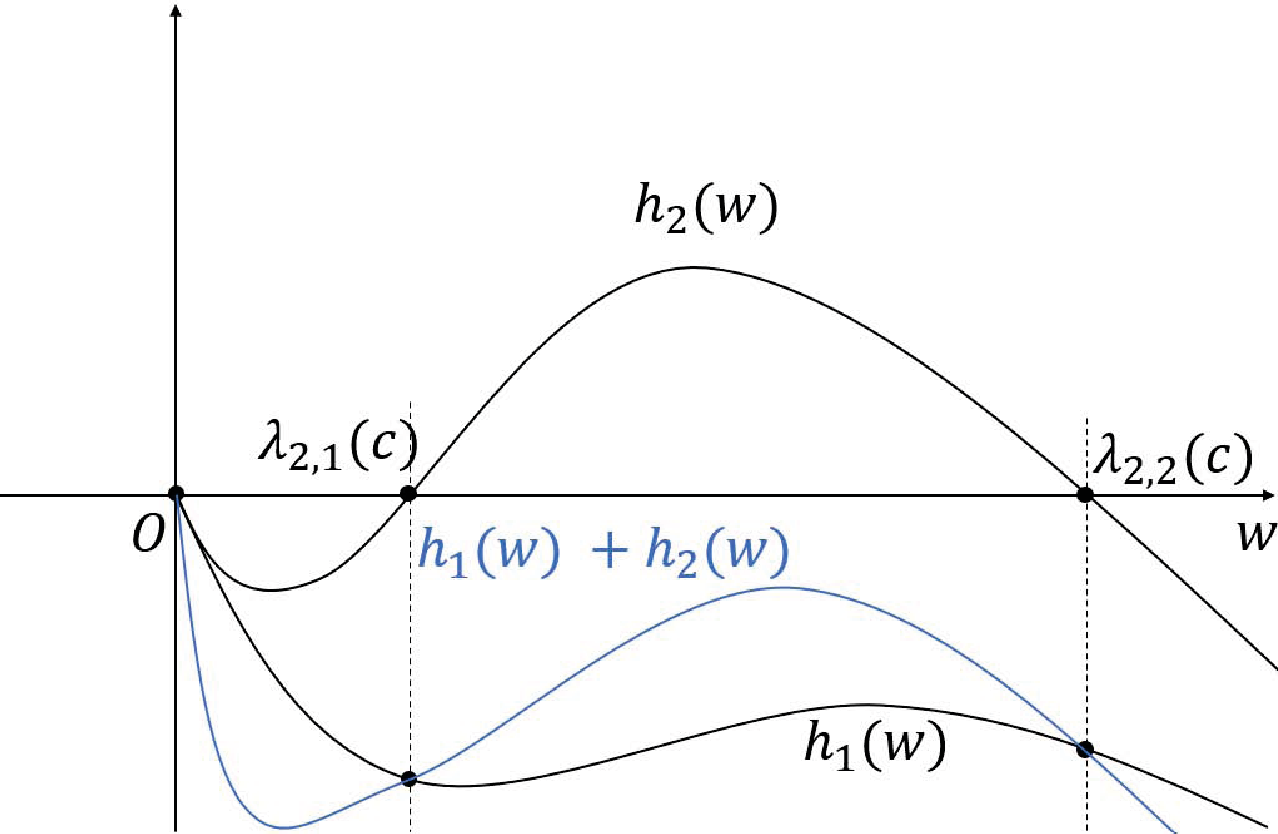}
        \caption{$g_1^*<c< g_2^*$ and $0<T\leq T^{***}$}
    \end{subfigure}
    \begin{subfigure}[b]{0.49\textwidth}
        \includegraphics[scale=0.32]{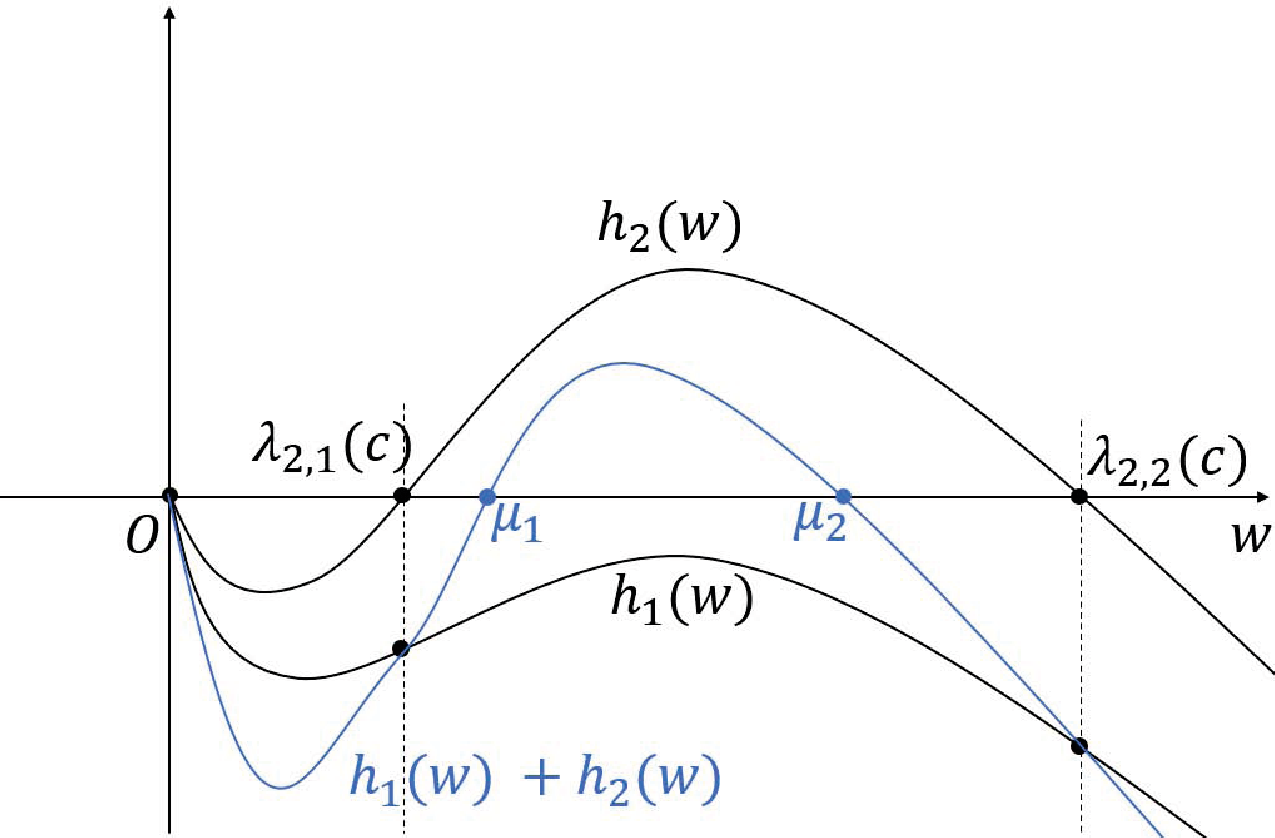}
        \caption{$g_1^*<c<g_2^*$ and $T^{***}<T<\overline T$}
    \end{subfigure}
    \begin{subfigure}[b]{0.49\textwidth}
        \includegraphics[scale=0.33]{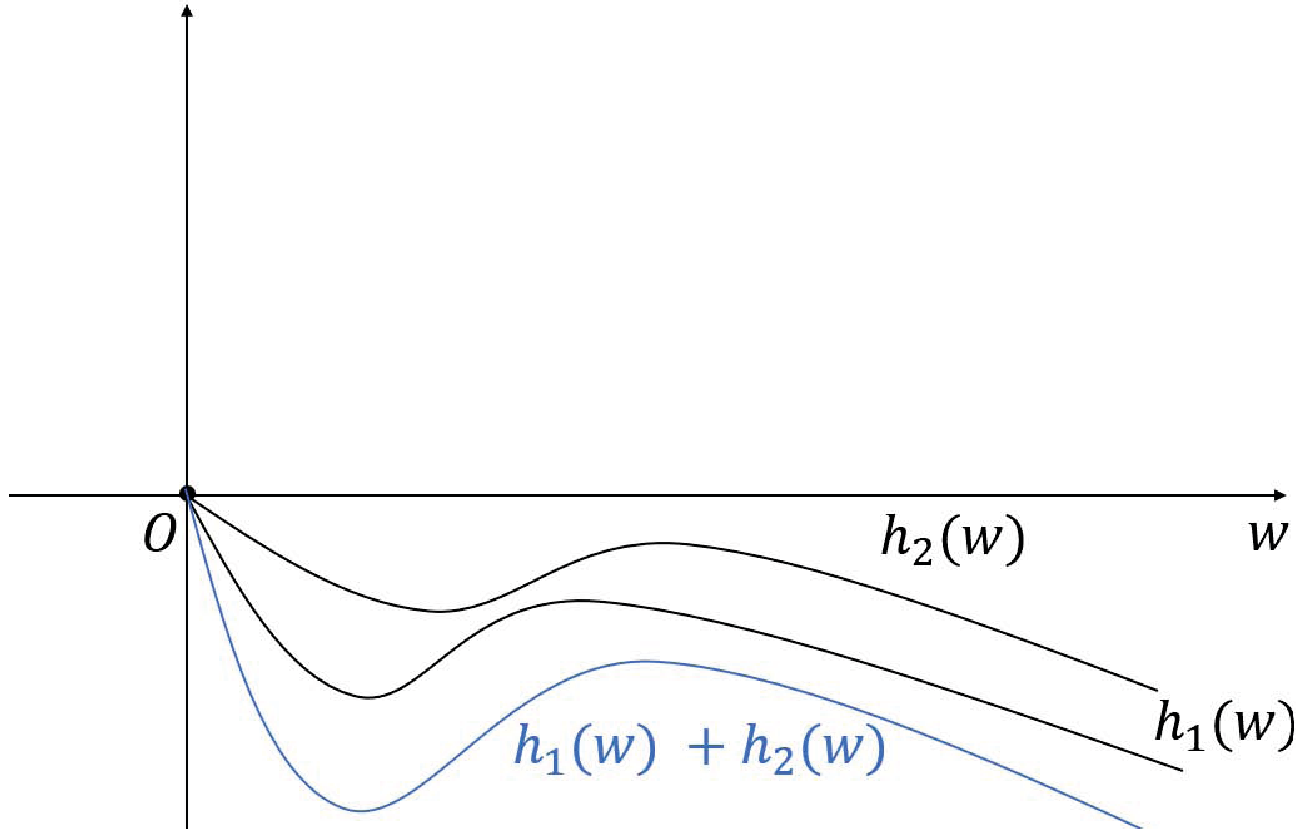}
        \caption{$c\geq g_2^*$}
    \end{subfigure}
    \caption{The relative positions of the zeros of $h_1$, $h_2$, and $h_1+h_2$ for different values of $c$ and $T$.}\label{Figures of application3.2}
\end{figure}

Now, we are able to characterize the limit cycles of equation \eqref{25} as follows.

\begin{proposition}\label{application3.2}
    For equation \eqref{25}, the constant limit cycle $w=0$ is stable. Moreover, the following statements hold.
    \begin{itemize}
        \item[(i)] If $0<c\leq g_{1}^{*}$, then equation \eqref{25} has exactly two non-constant limit cycles, where the lower one is unstable and the upper one is stable.
        \item[(ii)] If either $g_{1}^{*}<c<g_{2}^{*}$ and $0<T\leq T^{***}$, or $c\geq g_{2}^{*}$, then equation \eqref{25} has no non-constant limit cycles.
        \item[(iii)] If $g_{1}^{*}<c<g_{2}^{*}$ and $T^{***}<T<\overline{T}$, then equation \eqref{25} has at most two non-constant limit cycles. These limit cycles (if any) are located in $[0,1]\times(\lambda_{2,1},\lambda_{2,2})$.
    \end{itemize}
\end{proposition}
\begin{proof}
    First, the assertion for the stability of the limit cycle $w=0$ can be easily checked by using Proposition \ref{proposition2.6} and noting that
    \begin{equation*}
        h'_1(0)+h'_2(0)=-2q\big(\mu+\xi(p+1)c\big)-2(T-q)(\mu+\xi pc)<0.
    \end{equation*}

    Next we verify the statements one by one.

    (i) It is clear from Table \ref{Table of application3.2} and statement (ii) of Theorem \ref{th2} that equation \eqref{25} has at most two limit cycles.
    Furthermore, by Step 1 we get that:
    $h_{1}(\lambda_{2,1})<0$ and $h_{2}(\lambda_{2,1})=0$;
    $h_{1}(\lambda_{1,1})=0$ and $h_{2}(\lambda_{1,1})>0$;
    $h_{1}(\lambda_{1,2})=0$ and $h_{2}(\lambda_{1,2})>0$;
    $h_{1}(\lambda_{2,2})<0$ and $h_{2}(\lambda_{2,2})=0$.
    These yield an unstable limit cycle in $[0,1]\times (\lambda_{2,1},\lambda_{1,1})$ and a stable limit cycle in $[0,1]\times (\lambda_{1,2},\lambda_{2,2})$ of the equation. Hence, statement (i) holds.

    (ii) It is also a direct conclusion applying Table \ref{Table of application3.2} and statement (i.2) of Theorem \ref{th2}.

    (iii) Again due to Table \ref{Table of application3.2} and statement (i.2) of Theorem \ref{th2}, it is sufficient to estimate the number of limit cycles of  equation \eqref{25} in $[0,1]\times (\lambda_{2,1},\lambda_{2,2})$.
    Observe that $h'''_1<0$ and $h'''_2<0$ hold for the equation (see also \eqref{26} in Appendix \ref{statements of application3.2} for details). It follows from Proposition \ref{propostion2.5} that equation \eqref{25} has at most three limit cycles in $[0,1]\times\mathbb R^+_0$, where one is always at $w=0$.
    Consequently, the number of non-constant limit cycles in $[0,1]\times(\lambda_{2,1},\lambda_{2,2})$ is no more than two. The assertion is verified.
\end{proof}
\begin{proof}[Proof of Theorem \ref{subapplication2.2}:]
Since equation \eqref{mosquito} is reduced to equation \eqref{25} under assumption, the conclusion of the theorem follows directly from Proposition \ref{application3.2}.
\end{proof}

\section*{Acknowledgements}

We appreciate the anonymous referee for his/her comments and suggestions which help us to improve both mathematics and presentations of this paper.

 The first and third authors is supported by the NNSF of China (No. 12371183).
The second author are supported by NNSF of China (No. 12271212).

\section{Appendix}
\subsection{The third formula in Lemma \ref{han} and estimate in Proposition \ref{propostion2.5}}\label{A.1}
This subsection is devoted to verifying the derivative formula \eqref{han3} and then to proving Proposition \ref{propostion2.5}.
\begin{proof}[Proof of Lemma \ref{han}]
    We only need to prove statement (iii). By assumption, $f(t,x)$ is $C^3$ differentiable in $x$, and so does $x(t;0,x_0)$ in $x_0$. Differentiating equality \eqref{han2} yields
    \begin{equation*}
        \begin{aligned}
            P^{'''}(x_{0})=
            &P^{''}(x_{0})\int_{0}^{T}\frac{\partial^2 f}{\partial x^2}\big(t,x(t;0,x_{0})\big)\frac{\partial x}{\partial x_{0}}(t;0,x_{0})dt\\
            &+P^{'}(x_{0})\int_{0}^{T}\Bigg(\frac{\partial^3 f}{\partial x^3}\big(t,x(t;0,x_{0})\big)\bigg(\frac{\partial x}{\partial x_{0}}(t;0,x_{0})\bigg)^{2}\\
            &+\frac{\partial^2 f}{\partial x^2}\big(t,x(t;0,x_{0})\big)\frac{\partial^{2} x}{\partial x_{0}^{2}}(t;0,x_{0})\Bigg)dt.
        \end{aligned}
    \end{equation*}
    It is clear from \eqref{han2} that the first summand is equal to $(P''(x_0))^2/P'(x_0)$. With regard to the third summand note that, on account of the expression of $\frac{\partial x}{\partial x_0}(t;0,x_0)$ given in statement (ii),
    \begin{equation*}
        \begin{split}
\int_{0}^{T}&\frac{\partial^2 f}{\partial x^2}\big(t,x(t;0,x_{0})\big)\frac{\partial^{2} x}{\partial x_{0}^{2}}(t;0,x_{0})dt\\
            &=\int_{0}^{T}\Bigg(\frac{\partial^2 f}{\partial x^2}\big(t,x(t;0,x_{0})\big)\frac{\partial x}{\partial x_{0}}(t;0,x_{0})\int_{0}^{t}\frac{\partial^2 f}{\partial x^2}\big(s,x(s;0,x_{0})\big)\frac{\partial x}{\partial x_{0}}(s;0,x_{0})ds\Bigg)dt\\
            &=\frac{1}{2}\Bigg[\int_{0}^{T}\frac{\partial^2 f}{\partial x^2}\big(t,x(t;0,x_{0})\big)\frac{\partial x}{\partial x_{0}}(t;0,x_{0})dt\Bigg]^{2}
            =\frac{1}{2}\bigg(\frac{P''(x_0)}{P'(x_0)}\bigg)^2.
        \end{split}
    \end{equation*}
    Hence,
    \begin{equation*}
        P^{'''}(x_{0})=P^{'}(x_{0})\Bigg[\frac{3}{2}\bigg(\frac{P^{''}(x_{0})}{P^{'}(x_{0})}\bigg)^{2}+\int_{0}^{T}\frac{\partial^3 f}{\partial x^3}\big(t,x(t;0,x_{0})\big)\bigg(\frac{\partial x(t;0,x_{0})}{\partial x_{0}}\bigg)^{2}dt\Bigg].
    \end{equation*}
\end{proof}

\begin{proof}[Proof of Proposition \ref{propostion2.5}]
    For the case $k=1,2$ and the case $k=3$ with $f_i^{(3)}$'s being non-negative, the assertion follows directly from Lemma \ref{han} . Also the case when $k=3$ with $f_i^{(3)}$'s being non-positive can be reduced to that of being non-negative by taking into account the change of variable $x\mapsto -x$ for equation \eqref{equation1}.
\end{proof}

\subsection{Validity of statements (d)--(g) in Step 2 of the analysis of equation \eqref{YU11}}\label{statements of application3.1}
    We continue to use the notations $h_1$, $h_2$, $g^*$, $c^*$ and $T^*$ given in \eqref{YU11} and \eqref{thresholds1}. To show the validity of the statements, let us consider the function
    \begin{equation*}
        G_{T,c}(w)
        :=\frac{h_1(w)+h_2(w)}{2w}
        =
        \left(\frac{a\overline{T}}{w+c}-\xi T\right)w+(a-\mu)\left(T-T^*\right).
    \end{equation*}
    It is clear that on the interval $(0,+\infty)$, $G_{T,c}$ has the same zeros and the same sign as $h_{1}+h_{2}$. And it has at most two positive zeros because $(w+c)G_{T,c}(w)$ is a quadratic function. Moreover, from the facts stated in Step 1, we have that
    \begin{itemize}
      \item[($a$)] If $0<c< g^{*}$, then $G_{T,c}(\lambda_i)=\frac{h_{2}(\lambda_i)}{2\lambda_i}>0$ for $i=1,2$.
    \item[($b$)] If $c=g^{*}$, then $G_{T,c}(\lambda_*)=\frac{h_{2}(\lambda_*)}{2\lambda_*}>0$.
    \item[($c$)] $G<0$ on $[A,+\infty)$ for $c>0$.
    \end{itemize}

    Now we verify statements (d)--(g) one by one.

    (d) If $\overline{T}<T<T^{*}$ and $0<c< g^{*}$, then $G_{T,c}(0)=(a-\mu)\left(T-T^*\right)<0$. Taking statements ($a$) and ($c$) into account, we know that $G_{T,c}$ has exactly two zeros, denoted by $w=\mu_{1}$ and $w=\mu_{2}$, satisfying $\mu_1\in(0,\lambda_{1})$ and $\mu_2\in(\lambda_{2},A)$. In addition, it is easy to see that $G_{T,c}$ is negative on $(0,\mu_{1})\cup(\mu_{2},+\infty)$ and positive on $(\mu_{1},\mu_{2})$. Hence, statement (d.1) holds. Following the same argument but with statement ($b$) instead of statement ($a$), statement (d.2) is then obtained.

    (e) The conclusion is obvious taking into account statement ($c$).

    (f) According to statement ($c$), it is sufficient to verify the assertion for $w\in(0,A)$. To this end, we first get by a simple calculation that
    \begin{align*}
      \frac{\partial G_{T,c}(w)}{\partial T}=-\xi (w-A)>0
      ,\indent
      \frac{\partial G_{T,c}(w)}{\partial c}
      =-\overline T\left(\frac{a}{(w+c)^2}+\xi\right)<0.
    \end{align*}
    Thus, for $\overline T<T\leq T^{*}$, $c\geq c^{*}$ and $w\in(0,A)$, the function $G_{T,c}(w)$ is increasing in $T$ and decreasing in $c$, which implies that $G_{T,c}(w)\leq G_{T^*,c}(w)\leq G_{T^*,c^*}(w)$.

    On the other hand, observe that $G'_{T,c}(w)=\frac{ac\overline{T}}{(w+c)^{2}}-\xi T$. Therefore $G'_{T^*,c^*}$ is decreasing on $(0,A)$. From the definitions of $T^*$ and $c^*$, one has
    \begin{align}\label{A.2}
      G'_{T^*,c^*}(0)
      =
      \overline T\left(\frac{a}{c^*}-\xi \frac{a+\xi c^*}{a-\mu}\right)
      =
      -\frac{\overline T}{(a-\mu)c^*}\Big((\xi c^*)^{2}+a(\xi c^*)-a(a-\mu)\Big)
      =0.
    \end{align}
    Consequently, $G'_{T^*,c^*}<0$ (i.e. $G_{T^*,c^*}$ is decreasing) on $(0,A)$. We obtain
    $G_{T,c}(w)<G_{T^*,c^*}(w)<G_{T^*,c^*}(0)=0$
    for $\overline T<T\leq T^{*}$, $c\geq c^{*}$ and $w\in(0,A)$. The assertion follows.

    (g) First, if $T=T^{*}$ and $0<c<c^*$, then $G_{T,c}(0)=G_{T^*,c}(0)=0$. Using the expression of $G'_{T,c}$ and \eqref{A.2}, one can additionally check that $$c\cdot G'_{T,c}(0)=c\cdot G'_{T^*,c}(0)>c^*\cdot G'_{T^*,c^*}(0)=0,$$ which yields $G'_{T,c}(0)>0$. Also, if $T>T^{*}$, then $G_{T,c}(0)=(a-\mu)(T-T^*)>0$. These facts together with statement ($c$), ensure that $G_{T,c}$ has at least one zero on $(0,A)$ in both cases, with the total number being odd. Recall that the number of zeros of $G_{T,c}$ is at most two. This zero is unique and therefore the assertion immediately follows.

    \subsection{Validity of statements (f)--(h) in Step 2 of the analysis of equation \eqref{25}}\label{statements of application3.2}
    To start with, we know from \eqref{25} and a direct calculation that
    \begin{align}\label{26}
         h_1'''(w)=-\frac{12aq(p+1)c}{(w+(p+1)c)^4}<0,\
         h_2'''(w)=-\frac{12a(T-q)pc}{(w+pc)^4}<0,\indent w\in[0,+\infty).
    \end{align}
    Since $h_1+h_2$ always has a zero at $w=0$, it can have at most two positive zeros.

    Let us now verify statements (f)--(h) one by one.

        (f) If $0<c\leq g_{1}^{*}$, then according to the facts stated in Step 1, we have
         \begin{align*}
           h_{1}(\lambda_{2,1})+h_{2}(\lambda_{2,1})=h_{1}(\lambda_{2,1})<0\ \ \text{and}\ \
           h_{1}(\lambda_{1,1})+h_{2}(\lambda_{1,1})=h_{2}(\lambda_{1,1})>0.
         \end{align*}
        Thus, $h_{1}+h_{2}$ has a zero on the interval $(\lambda_{2,1},\lambda_{1,1})$, denoted by $w=\mu_1$.
        Similarly, another zero of $h_1+h_2$ is guaranteed on $(\lambda_{1,2},\lambda_{2,2})$, denoted by $w=\mu_2$. These two zeros are therefore the only ones of the function. And one can easily see by \eqref{25} that $h_{1}+h_{2}$ is negative on $(0,\mu_{1})\cup(\mu_{2},+\infty)$ and positive on $(\mu_{1},\mu_{2})$. Statement (f) holds.

        (g) If $g_{1}^{*}<c< g_{2}^{*}$, then again due to the facts in Step 1, $h_{1}+h_2$ is negative on $(0,\lambda_{2,1})\cup(\lambda_{2,2},+\infty)$ and therefore can only
        have zeros on $(\lambda_{2,1},\lambda_{2,2})$. Statement (g.2) follows.
        Moreover, we also know from Step 1 that $h_2$ is positive on $(\lambda_{2,1},\lambda_{2,2})$ and
        \begin{align*}
         h_1(w)+h_2(w)
         &=\frac{h_2(w)}{T-q}\left(T-q+\frac{(T-q) h_1(w)}{ h_2(w)}\right)\\
         &=\frac{h_2(w)}{T-q}\cdot
           \left(T-q\left(1-\frac{w+pc}{w+(p+1)c}\cdot\frac{\hat h_1(w)}{\hat h_2(w)}\right)\right).
        \end{align*}
        Thus, from definition \eqref{threshold3}, $h_{1}+h_{2}\leq 0$ holds on $(\lambda_{2,1},\lambda_{2,2})$ if additionally $T\leq T^{***}$. Statement (g.1) is true.

        (h) In this case where $c\geq g_{2}^{*}$, the conclusion is immediately obtained from the fact that both $h_{1}$ and $h_{2}$ are negative on $(0,+\infty)$.

\end{document}